\documentclass[11pt]{article}

\usepackage[a4paper,margin=29mm]{geometry}
\usepackage{amsmath,amssymb,amsthm,mathtools}
\usepackage{array,booktabs,microtype}
\usepackage[hidelinks]{hyperref}

\allowdisplaybreaks[2]
\numberwithin{equation}{section}

\newtheorem{theorem}{Theorem}[section]
\newtheorem{lemma}[theorem]{Lemma}
\newtheorem{proposition}[theorem]{Proposition}

\theoremstyle{remark}
\newtheorem{remark}[theorem]{Remark}

\newtheorem{theoremA}{Theorem}

\newtheorem{theoremB}{Theorem}

\newcommand{\ICG}{\operatorname{ICG}}
\newcommand{\Energy}{\mathcal E}
\newcommand{\Dstar}{\mathcal D^*}
\newcommand{\one}{\mathbf 1}
\newcommand{\R}{\mathbb R}
\newcommand{\sgn}{\operatorname{sgn}}
\newcommand{\tr}{\operatorname{tr}}
\newcommand{\Span}{\operatorname{span}}
\newcommand{\eps}{\varepsilon}

\title{Ramanujan Transforms and Maximum Energy of Odd-Total-Exponent Integral Circulants}
\author{Jianwei Jiang\thanks{Email: \texttt{jianweijiangwf@sina.com}}\quad
Chunhua Yang\thanks{Corresponding author. Email: \texttt{chunhuayangwf@sina.com}}\\
\small School of Mathematics and Statistics, Weifang University,\\
\small Weifang, Shandong, China}
\date{}

\begin{document}

\maketitle

\begin{abstract}
For a finite graph $G$, its energy is the sum of the absolute values of its adjacency eigenvalues. Let $n=\prod_{\nu=1}^{m}p_\nu^{a_\nu}$, where the $p_\nu$ are distinct odd primes, $a_\nu\geq1$, $\sum_{\nu=1}^{m}a_\nu$ is odd, and $p_\nu\geq5$ whenever $a_\nu$ is odd. Among all nonempty sets of proper divisors of $n$, the energy of the corresponding integral circulant graph is uniquely maximised by the checkerboard divisor set consisting of those divisors $\prod_{\nu=1}^{m}p_\nu^{i_\nu}$ with $0\leq i_\nu\leq a_\nu$ and $\sum_{\nu=1}^{m}i_\nu$ even. The maximum energy is
\[
\mathcal E_{\max}(n)
=
\frac12\left[
n+\prod_{\nu=1}^{m}d_{a_\nu}(p_\nu)
\right],
\qquad
d_k(x)
=
(2k+1)x^k
+
4\sum_{j=0}^{k-1}(-1)^{k-j}(j+1)x^j.
\]

For the proof, we use weighted prime-power Ramanujan transforms, which give an exact energy model in tensor form and extend to a polynomial matrix family $T_k(x)$ for real $x>1$. A semidefinite parity theorem shows that the matrix obtained from $T_k(x)$ by an explicit positive diagonal normalisation is a real symmetric contraction whose unique eigenvalue of modulus one is the simple eigenvalue $(-1)^k$; this holds for even $k\geq2$ and $x\geq3$, and for odd $k\geq1$ and $x\geq5$. Tensorising this spectral structure and centring the binary divisor tensor yield a sharp $\ell_1$-norm bound for the corresponding tensor operator on sign vectors. In particular, for $n=p^2q^3$ with distinct odd primes $p,q$ and $q\geq5$, this proves Rold\'an's conjectured maximality in this range.
\end{abstract}



\section{Introduction}

For a finite graph $G$ with adjacency eigenvalues $\lambda_1,\ldots,\lambda_n$, its energy is $\Energy(G)=\sum_{j=1}^n|\lambda_j|$. This invariant was introduced by Gutman \cite{gutman1978energy}. An integral circulant graph, or gcd-graph, is determined by an integer $n$ and a set $\mathcal D$ of proper divisors of $n$: its vertex set is $\mathbb Z/n\mathbb Z$, and distinct vertices $u,v$ are adjacent when $\gcd(u-v,n)\in\mathcal D$. We write $\ICG_n(\mathcal D)$ for this graph. The integral circulant graphs are precisely the integral graphs among circulant graphs in the standard divisor-set description; see So \cite{so2006integral}. We consider only nonempty $\mathcal D$.

The energy of integral circulant graphs has been studied in a number of arithmetic and extremal settings. Ili\'c and Ba\v{s}i\'c obtained closed-form energy formulas for several classes and studied minimum energy for even-order integral circulant graphs \cite{ilicbasic2011energy}. For prime-power orders, Sander and Sander derived an explicit energy formula and studied the corresponding extremal problem \cite{sander2011energy}. Subsequent work gave bounds and structural information for maximal energy \cite{sander2011maximal}, including the fixed-size divisor-set problem \cite{sander2012classes}, culminating in the exact determination of all maximising divisor sets and the maximal energy \cite{sander2013exact}. Beyond prime powers, Le and Sander developed a multiplicative Ramanujan-sum framework for integral circulant graphs \cite{le2012convolutions} and used multiplicative divisor sets to obtain bounds on the extremal energies of integral circulant graphs of arbitrary order \cite{le2012extremal}.

More recently, Rold\'an considered the two-prime order $p^2q^3$. He singled out the divisor set $\{1,p^2,pq,q^2,p^2q^2,pq^3\}$, identified a Kronecker separation in its eigenvalue structure, obtained its energy in closed form, and conjectured that it uniquely maximises the energy for every pair of distinct odd primes \cite[Theorems~1.1--1.2 and Conjecture~1.3]{roldan2026graph}. This raises two related questions: whether the checkerboard pattern is globally extremal, and whether it persists beyond the exponent pair $(2,3)$ and beyond two prime factors.

For orders with several prime factors, a general divisor set need not separate into independent choices on the prime-power coordinates. A Kronecker factorisation of the spectral transform therefore leaves a coupled optimisation over binary divisor tensors. The matrix inequality developed below controls this coupled problem. We first state our graph-theoretic maximum-energy result for odd total exponent.

\begin{theoremA}\label{thm:main}
Let $m\geq1$, let $p_1,\ldots,p_m$ be distinct odd primes, and let
$a_1,\ldots,a_m\geq1$ be integers.  Assume that
$\sum_{\nu=1}^{m}a_\nu$ is odd and that $p_\nu\geq5$ whenever $a_\nu$ is
odd.  Put $n=\prod_{\nu=1}^{m}p_\nu^{a_\nu}$.  Among all nonempty sets of
proper divisors of $n$, the energy of $\ICG_n(\mathcal D)$ is uniquely
maximised by
\begin{equation}
 \Dstar=\left\{\prod_{\nu=1}^{m}p_\nu^{i_\nu}:
  0\leq i_\nu\leq a_\nu,\quad
  \sum_{\nu=1}^{m}i_\nu\equiv0\pmod2\right\}.
\end{equation}
Writing $\Energy_{\max}(n)$ for this maximum, we have
\begin{equation}\label{eq:main-energy}
 \Energy_{\max}(n)=\frac12\left[n+
  \prod_{\nu=1}^{m}d_{a_\nu}(p_\nu)\right],\qquad
 d_k(x)=(2k+1)x^k+4\sum_{j=0}^{k-1}(-1)^{k-j}(j+1)x^j.
\end{equation}
\end{theoremA}

For $m=1$, this recovers the odd-exponent prime-power case of Sander and Sander \cite[Theorem~1.1(i)]{sander2013exact} in the stated prime range. For $m\geq2$, the maximum is taken over all nonempty sets of proper divisors, with no multiplicativity assumption. In particular, the case $n=p^2q^3$ establishes Rold\'an's conjectured maximality for $q\geq5$ and recovers his energy formula.

The main ingredient in the proof of Theorem~\ref{thm:main} is the following semidefinite parity theorem. For an integer $k\geq1$ and a real parameter $x>1$, let $T_k(x)$ be the weighted prime-power Ramanujan transform defined in Section~\ref{sec:transforms}. At prime values of $x$, its entries are given by \eqref{eq:T-ramanujan-definition}; the polynomial expressions in \eqref{eq:T-interior}--\eqref{eq:T-boundary} define the extension to real $x>1$. Let $s_k=(1,-1,\ldots,(-1)^k)^{\mathsf T}$ be the parity vector, and let $\Delta_k(x)$ be the diagonal matrix in \eqref{eq:delta-definition}.

\begin{theoremB}\label{thm:parity}
For even $k\geq2$ and $x\geq3$,
\[
 \Delta_k(x)-T_k(x)\succeq0,\qquad
 \ker(\Delta_k(x)-T_k(x))=\Span\{s_k\},\qquad
 \Delta_k(x)+T_k(x)\succ0.
\]
For odd $k\geq1$ and $x\geq5$,
\[
 \Delta_k(x)+T_k(x)\succeq0,\qquad
 \ker(\Delta_k(x)+T_k(x))=\Span\{s_k\},\qquad
 \Delta_k(x)-T_k(x)\succ0.
\]
\end{theoremB}

The one-dimensional kernel in Theorem~\ref{thm:parity} is used to characterise equality in the tensor bound. In Proposition~\ref{prop:centered-tensor}, these factor bounds tensorise to a sharp $\ell_1$-norm bound for the weighted tensor operator on sign vectors, with equality only at the two checkerboard sign vectors. The proper-divisor condition selects the one corresponding to an admissible divisor tensor. Verifying equality in the triangle inequality used in the energy model then yields Theorem~\ref{thm:main}.

To place Theorem~\ref{thm:parity} in the context of earlier Ramanujan-matrix work, Fowler, Garcia and Karaali \cite{fowler2014supercharacters} studied divisor-indexed Ramanujan matrices through supercharacter theory, obtaining reciprocity and orthogonality relations, a unitary involution, and a Kronecker factorisation over coprime factors. Theorem~\ref{thm:parity} concerns a different diagonal normalisation: it yields parity-dependent semidefinite bounds with an exact one-dimensional kernel and complementary strict positive definiteness.

The proof of Theorem~\ref{thm:parity} combines a parity-independent congruence reduction with a common nonterminal Schur recurrence, followed by separate endpoint calculations for odd and even $k$.

\section{Ramanujan transforms and the tensor energy model}
\label{sec:transforms}

\subsection{The weighted prime-power Ramanujan transform}

For integers $\ell\geq1$ and $u$, write
\[
 c_\ell(u)=\sum_{\substack{1\leq h\leq\ell\\(h,\ell)=1}}
          \exp(2\pi \mathrm{i}hu/\ell)
\]
for the Ramanujan sum, where $\mathrm{i}^2=-1$.  For standard properties used below, see, for example,
Apostol \cite[Sections~8.3--8.4, pp.~160--164]{apostol1976analytic}.  If $\ell$ is prime, then for
integers $a\geq1$ and $i\geq0$,
\begin{equation}\label{eq:ramanujan-prime-power}
 c_{\ell^a}(\ell^i)=
 \begin{cases}
  \ell^{a-1}(\ell-1),&a\leq i,\\
  -\ell^{a-1},&a=i+1,\\
  0,&a\geq i+2,
 \end{cases}
 \qquad c_1(\ell^i)=1.
\end{equation}

For a prime $x$, an integer $k\geq1$, and $0\leq i,j\leq k$, the weighted
prime-power Ramanujan transform is initially defined by
\begin{equation}\label{eq:T-ramanujan-definition}
 (T_k(x))_{ij}=\varphi(x^{k-i})c_{x^{k-j}}(x^i).
\end{equation}
Here $\varphi$ denotes Euler's totient function. The weighting is chosen to
incorporate the spectral multiplicities in the energy model below: the factor
$c_{x^{k-j}}(x^i)$ records the Ramanujan-sum contribution, while
$\varphi(x^{k-i})$ provides the corresponding multiplicity. Substituting the
prime-power formula \eqref{eq:ramanujan-prime-power} into
\eqref{eq:T-ramanujan-definition} yields the explicit polynomial expressions
in the following proposition. We use these expressions to define $T_k(x)$
for every real $x>1$.

\begin{proposition}
Let $T_k(x)=(t_{ij})_{0\leq i,j\leq k}$.  If $i,j<k$, then
\begin{equation}\label{eq:T-interior}
 t_{ij}=\begin{cases}
  0,&i+j\leq k-2,\\[1mm]
  -x^{k-1}(x-1),&i+j=k-1,\\[1mm]
  x^{2k-i-j-2}(x-1)^2,&i+j\geq k.
 \end{cases}
\end{equation}
The boundary entries are
\begin{equation}\label{eq:T-boundary}
 t_{i,k}=t_{k,i}=x^{k-i-1}(x-1)\quad(0\leq i<k),
 \qquad t_{kk}=1.
\end{equation}
In particular,
\begin{equation}\label{eq:T-symmetric-rowsum}
 T_k(x)^{\mathsf T}=T_k(x),\qquad T_k(x)\one=x^ke_k,
\end{equation}
where $e_k$ is the last standard basis vector. Here $\one_d\in\mathbb R^d$
is the all-one column vector; we omit the subscript when the dimension is clear.
\end{proposition}

\begin{proof}
The three alternatives in \eqref{eq:T-interior} correspond respectively to
$k-j\geq i+2$, $k-j=i+1$, and $k-j\leq i$ in
\eqref{eq:ramanujan-prime-power}.  Multiplying by
$\varphi(x^{k-i})=x^{k-i-1}(x-1)$ gives the displayed entries whenever
$i<k$. For the boundary entries, if
$0\leq i<k$, then
\[
 t_{i,k}=\varphi(x^{k-i})c_1(x^i)=x^{k-i-1}(x-1),
 \qquad
 t_{k,i}=c_{x^{k-i}}(x^k)=\varphi(x^{k-i}),
\]
and $t_{kk}=1$.  Thus \eqref{eq:T-boundary} holds, and together with
\eqref{eq:T-interior} it shows that $T_k(x)$ is symmetric.

For the row sum, we use the standard Ramanujan-sum identity
\[
 \sum_{d\mid n}c_d(u)=
 \begin{cases}
  n,&n\mid u,\\
  0,&n\nmid u.
 \end{cases}
\]
As $j$ ranges from $0$ to $k$, the moduli $x^{k-j}$ run through all
divisors of $x^k$.  Thus, for prime $x$, the $i$th row sum is
\[
 \varphi(x^{k-i})\sum_{d\mid x^k}c_d(x^i),
\]
and hence, for every prime $x$,
\[
 T_k(x)\mathbf 1=x^k e_k.
\]
Each coordinate on both sides is a polynomial in $x$.  Since this identity
holds for infinitely many prime values of $x$, it is a polynomial identity
and therefore holds for every real $x>1$.
\end{proof}

\subsection{The exact energy model}

Let $m\geq1$, let $p_1,\ldots,p_m$ be distinct primes, and let
$a_1,\ldots,a_m\geq1$ be integers.  Put
$n=\prod_{\nu=1}^{m}p_\nu^{a_\nu}$, and write
\begin{equation}\label{eq:index-set}
 \mathcal I=\left\{\mathbf i=(i_1,\ldots,i_m)\in\mathbb Z^m:
 0\leq i_\nu\leq a_\nu\text{ for }1\leq\nu\leq m\right\}.
\end{equation}
Set $N=|\mathcal I|=\prod_{\nu=1}^{m}(a_\nu+1)$ and
$\mathbf a=(a_1,\ldots,a_m)$.  Encode a divisor set by the binary divisor tensor
\begin{equation}\label{eq:X-definition}
 X=(x_{\mathbf i})_{\mathbf i\in\mathcal I},\qquad
 x_{\mathbf i}\in\{0,1\},\qquad
 x_{\mathbf i}=1\Longleftrightarrow
 \prod_{\nu=1}^{m}p_\nu^{i_\nu}\in\mathcal D.
\end{equation}
The proper-divisor condition is exactly $x_{\mathbf a}=0$. We call $X$
admissible if $X\neq0$ and $x_{\mathbf a}=0$, equivalently if it encodes a
nonempty set of proper divisors of $n$.
We vectorise $X$ with the first coordinate varying fastest: the entry
$x_{\mathbf i}$ occupies position
\begin{equation}\label{eq:vec-convention}
 \iota(\mathbf i)=\sum_{\nu=1}^{m}i_\nu
                   \prod_{\mu=1}^{\nu-1}(a_\mu+1)
 \quad\text{in}\quad \mathbf x=\operatorname{vec}(X)\in\R^N,
\end{equation}
where positions are numbered from $0$ and the empty product is $1$. Equivalently, \(i_1\) varies first through \(0,\ldots,a_1\); after one complete cycle of \(i_1\), the coordinate \(i_2\) is increased by one, and the same convention is applied successively to the remaining coordinates.
We also use the multi-index $\mathbf i$ to denote this vector coordinate.
For a real vector $v$, write $\lVert v\rVert_1=\sum_i|v_i|$ and $\lVert v\rVert_2=(v^{\mathsf T}v)^{1/2}$.
Set
\begin{equation}\label{eq:tensor-operator}
 M_\nu=T_{a_\nu}(p_\nu)\quad(1\leq\nu\leq m),\qquad
 K=M_m\otimes\cdots\otimes M_1.
\end{equation}
We use the little-endian multi-index convention in
\eqref{eq:vec-convention}, in which the first coordinate varies fastest.
Accordingly, the factor order in \eqref{eq:tensor-operator} matches this
vectorisation convention. For $\mathbf i=(i_1,\ldots,i_m)\in\mathcal I$,
let $e_{\mathbf i}\in\mathbb R^N$ be the standard basis vector at position
$\iota(\mathbf i)$ in \eqref{eq:vec-convention}, and let
$e_{i_\nu}\in\mathbb R^{a_\nu+1}$ denote the standard basis vector indexed
by $i_\nu$. Under the ordering in \eqref{eq:vec-convention},
\[
 e_{\mathbf i}=e_{i_m}\otimes\cdots\otimes e_{i_1}.
\]
Since $K_{\mathbf u,\mathbf i}$ is the $\mathbf u$-coordinate of
$K e_{\mathbf i}$, applying \eqref{eq:finite-action} and then
\eqref{eq:finite-coordinate} from Appendix~\ref{app:finite-kron} gives
\begin{equation}\label{eq:kronecker-coordinate}
 K_{\mathbf u,\mathbf i}
 =\prod_{\nu=1}^{m}(M_\nu)_{u_\nu,i_\nu}
 \qquad (\mathbf u,\mathbf i\in\mathcal I).
\end{equation}

\begin{proposition}\label{prop:energy-model}
With the notation in \eqref{eq:index-set}--\eqref{eq:tensor-operator},
\begin{equation}\label{eq:exact-matrix-energy}
 \Energy\bigl(\ICG_n(\mathcal D)\bigr)=\lVert K\mathbf x\rVert_1.
\end{equation}
\end{proposition}

\begin{proof}
By the work of So \cite{so2006integral}, the adjacency eigenvalues of
$\ICG_n(\mathcal D)$ can be expressed as
\begin{equation}\label{eq:eigenvalue-ramanujan}
 \lambda_t=\sum_{d\in\mathcal D}c_{n/d}(t)
 =\sum_{\mathbf i\in\mathcal I}x_{\mathbf i}
   \prod_{\nu=1}^{m}c_{p_\nu^{a_\nu-i_\nu}}(t),
 \qquad t\in\mathbb Z/n\mathbb Z.
\end{equation}
The second equality follows from \eqref{eq:X-definition} and the
multiplicativity of Ramanujan sums for coprime moduli.

Choose an integer representative of $t$.  The truncated valuations below
are independent of this choice.  Here $v_{p_\nu}$ denotes the usual
$p_\nu$-adic valuation, with $v_{p_\nu}(0)=\infty$.  Put
\begin{equation}\label{eq:valuation-vector}
 \mathbf u=(u_1,\ldots,u_m),\qquad
 u_\nu=\min\{v_{p_\nu}(t),a_\nu\}.
\end{equation}
In particular, $t=0$ corresponds to $\mathbf u=\mathbf a$.
If $u_\nu<a_\nu$, multiplication by a unit modulo a power of $p_\nu$
permutes the reduced residue classes in the Ramanujan sum.  If
$u_\nu=a_\nu$, both $t$ and $p_\nu^{a_\nu}$ are divisible by every modulus
$p_\nu^{a_\nu-i_\nu}$ occurring in \eqref{eq:eigenvalue-ramanujan}.
Consequently, for every $\nu$ and $0\leq i_\nu\leq a_\nu$,
\begin{equation}\label{eq:ramanujan-layer}
 c_{p_\nu^{a_\nu-i_\nu}}(t)
 =c_{p_\nu^{a_\nu-i_\nu}}(p_\nu^{u_\nu}).
\end{equation}
Thus the eigenvalue is constant on the set of spectral parameters with the
same truncated valuation vector $\mathbf u$; we call this the
$\mathbf u$-layer and write $\lambda_{\mathbf u}$ for its common eigenvalue.

If $u_\nu<a_\nu$, the number of residue classes modulo $p_\nu^{a_\nu}$
with valuation $u_\nu$ is $\varphi(p_\nu^{a_\nu-u_\nu})$.  If
$u_\nu=a_\nu$, there is only the zero class, whose count is
$1=\varphi(1)$.  By the Chinese remainder theorem, the number of spectral
parameters in the $\mathbf u$-layer is therefore
\begin{equation}\label{eq:layer-multiplicity}
 w_{\mathbf u}=\prod_{\nu=1}^{m}
                \varphi(p_\nu^{a_\nu-u_\nu}).
\end{equation}
Different layers may have the same eigenvalue; their contributions to the
energy are added separately.

By \eqref{eq:T-ramanujan-definition} and \eqref{eq:kronecker-coordinate}, the coordinate of \(K\mathbf x\) indexed by \(\mathbf u\) is
\begin{align}
 (K\mathbf x)_{\mathbf u}
 &=\sum_{\mathbf i\in\mathcal I}x_{\mathbf i}
   \prod_{\nu=1}^{m}\left[
    \varphi(p_\nu^{a_\nu-u_\nu})
    c_{p_\nu^{a_\nu-i_\nu}}(p_\nu^{u_\nu})\right]\notag\\
 &=w_{\mathbf u}\lambda_{\mathbf u},
 \label{eq:tensor-layer-entry}
\end{align}
where the last equality uses \eqref{eq:eigenvalue-ramanujan},
\eqref{eq:ramanujan-layer}, and \eqref{eq:layer-multiplicity}.
Every spectral parameter determines exactly one vector in
\eqref{eq:valuation-vector}, so the layers partition $\mathbb Z/n\mathbb Z$.
Since $w_{\mathbf u}>0$, \eqref{eq:tensor-layer-entry} gives
\[
 \Energy\bigl(\ICG_n(\mathcal D)\bigr)
 =\sum_{\mathbf u\in\mathcal I}w_{\mathbf u}|\lambda_{\mathbf u}|
 =\sum_{\mathbf u\in\mathcal I}|(K\mathbf x)_{\mathbf u}|
 =\lVert K\mathbf x\rVert_1,
\]
which proves \eqref{eq:exact-matrix-energy}.
\end{proof}

\section{A semidefinite parity theorem for weighted prime-power Ramanujan transforms}
\label{sec:parity}

This section develops the semidefinite structure of the weighted prime-power
Ramanujan transforms $T_k(x)$ and proves Theorem~\ref{thm:parity}. Throughout
this section, dependence on $x$ is suppressed from $T_k(x)$, $\Delta_k(x)$,
and related quantities when no confusion can arise.

\subsection{The parity vector, diagonal weights, and trace formula}

For $i\geq0$ define the alternating polynomial
\begin{equation}\label{eq:R-definition}
 R_i(x)=x^i-x^{i-1}+\cdots+(-1)^i
       =\frac{x^{i+1}+(-1)^i}{x+1}.
\end{equation}
For $i\geq1$, these polynomials satisfy
\begin{equation}\label{eq:R-recurrence}
 R_i(x)=xR_{i-1}(x)+(-1)^i=x^i-R_{i-1}(x).
\end{equation}
Let $s_k=(1,-1,\ldots,(-1)^k)^{\mathsf T}$.  The diagonal weights below
are chosen to match the action of $T_k(x)$ on the parity vector $s_k$.
Define $\Delta_k(x)=\operatorname{diag}(\delta_0,\ldots,\delta_k)$ by
\begin{equation}\label{eq:delta-definition}
 \delta_i=2x^{k-i-1}(x-1)R_i(x)\quad(0\leq i<k),
 \qquad
 \delta_k=2R_k(x)-x^k.
\end{equation}
The last entry can also be written as
\begin{equation}\label{eq:delta-last}
 \delta_k=\frac{x^k(x-1)+2(-1)^k}{x+1}.
\end{equation}
Since $R_i(x)>0$ for $x>1$, \eqref{eq:delta-definition} gives
$\delta_i>0$ for $0\leq i<k$.  In the parameter ranges of
Theorem~\ref{thm:parity}, \eqref{eq:delta-last} also gives $\delta_k>0$.
Hence $\Delta_k(x)$ is positive diagonal in those ranges.

\begin{lemma}
For every $k\geq1$ and $x>1$,
\begin{equation}\label{eq:parity-identity}
 T_k(x)s_k=(-1)^k\Delta_k(x)s_k.
\end{equation}
\end{lemma}

\begin{proof}
For $i<k$, the nonzero terms in row $i$ of
\eqref{eq:T-interior}--\eqref{eq:T-boundary} give
\begin{align*}
 (T_ks_k)_i={}&-(-1)^{k-1-i}x^{k-1}(x-1)\\
 &+\sum_{j=k-i}^{k-1}(-1)^j
   x^{2k-i-j-2}(x-1)^2
   +(-1)^kx^{k-i-1}(x-1).
\end{align*}
For $i\geq1$, after factoring $(-1)^{k+i}x^{k-i-1}(x-1)$, the
remaining expression is $x^i+(x-1)R_{i-1}(x)+(-1)^i=2R_i(x)$, where the
equality follows from \eqref{eq:R-recurrence}.  For $i=0$ the sum is empty,
and the remaining expression is $2=2R_0(x)$.  Hence, by
\eqref{eq:delta-definition}, $(T_ks_k)_i=(-1)^{k+i}\delta_i$.  In the last row,
\eqref{eq:T-boundary} and \eqref{eq:R-definition} give
\[
 (T_ks_k)_k=(x-1)\sum_{j=0}^{k-1}(-1)^j x^{k-j-1}+(-1)^k
 =(x-1)R_{k-1}(x)+(-1)^k.
\]
By \eqref{eq:R-recurrence},
\[
 2R_k(x)=\bigl(xR_{k-1}(x)+(-1)^k\bigr)
          +\bigl(x^k-R_{k-1}(x)\bigr)
 =x^k+(x-1)R_{k-1}(x)+(-1)^k,
\]
so $(T_ks_k)_k=2R_k(x)-x^k=\delta_k$ by
\eqref{eq:delta-definition}.  These are exactly the coordinates of
\eqref{eq:parity-identity}.
\end{proof}

The trace of $\Delta_k(x)$ recovers the polynomial $d_k(x)$ defined in
\eqref{eq:main-energy}:
\begin{equation}\label{eq:dk-trace}
 d_k(x)=\tr\Delta_k(x).
\end{equation}
Summing \eqref{eq:delta-definition} and using \eqref{eq:R-definition} gives
\begin{align}
 d_k(x)
 &=\frac{(2k+1)x^k(x-1)+2(x-1)R_{k-1}(x)+2(-1)^k}{x+1}
 \notag\\
 &=(2k+1)x^k+4\sum_{j=0}^{k-1}(-1)^{k-j}(j+1)x^j.
 \label{eq:dk-formula}
\end{align}
The second line follows by polynomial division, or directly by comparing
successive coefficients. 

\subsection{A parity-independent congruence reduction}

We next give a congruence reduction that applies to both parity cases in
Theorem~\ref{thm:parity}.  The aim is to transform $\Delta_k\pm T_k$ by congruence
into sparse forms suited to the opposite-endpoint block elimination used below.
Let $P_k$ be the reversal matrix of order $k+1$ and put
\begin{equation}
 W_k=\operatorname{diag}\bigl(1,x^0(x-1),x^1(x-1),\ldots,
                              x^{k-1}(x-1)\bigr).
\end{equation}
Since $x>1$, $W_k$ is invertible; the reversal matrix $P_k$ is a
permutation matrix and hence is invertible as well.  A direct substitution
of the entries of $T_k$ gives
\begin{equation}\label{eq:HT-congruence}
 W_k^{-1}P_kT_kP_kW_k^{-1}=H_k.
\end{equation}
Indeed, the reversal sends the $(i,j)$ entry to $t_{k-i,k-j}$.  Thus, for
$i,j\geq1$, the $(i,j)$ entry on the left-hand side of
\eqref{eq:HT-congruence} is
\[
 \frac{t_{k-i,k-j}}{x^{i+j-2}(x-1)^2}.
\]
If $i+j\leq k$, $i+j=k+1$, or $i+j\geq k+2$, the numerator is respectively
\[
 x^{i+j-2}(x-1)^2,\qquad
 -x^{k-1}(x-1),\qquad
 0,
\]
by \eqref{eq:T-interior}.  The cases $i=0$ or $j=0$ follow directly from the
boundary entries in \eqref{eq:T-boundary}.  Hence, for $0\leq i,j\leq k$,
\begin{equation}\label{eq:H-entries}
 (H_k)_{ij}=\begin{cases}
  1,&i+j\leq k,\\[1mm]
  -1/(x-1),&i+j=k+1,\\[1mm]
  0,&i+j\geq k+2.
 \end{cases}
\end{equation}
Similarly,
\begin{equation}\label{eq:BDelta-congruence}
 W_k^{-1}P_k\Delta_kP_kW_k^{-1}
 =B_k=\operatorname{diag}(b_0,\ldots,b_k).
\end{equation}
For $i\geq0$, write
\begin{equation}\label{eq:F-definition}
 F_i(x)=2R_i(x)-x^i
       =\frac{x^i(x-1)+2(-1)^i}{x+1}.
\end{equation}
Then
\begin{equation}\label{eq:b-entries}
 b_0=F_k,\qquad
 b_i=\frac{2R_{k-i}}{x^{i-1}(x-1)}\quad(1\leq i\leq k).
\end{equation}
For $i\geq1$, the latter formula follows directly from
\eqref{eq:delta-definition}:
\[
 b_i=\frac{\delta_{k-i}}{x^{2i-2}(x-1)^2}
     =\frac{2R_{k-i}}{x^{i-1}(x-1)},
\]
while $b_0=\delta_k=F_k$.

Let $C_k$ be the $(k+1)\times(k+1)$ upper bidiagonal first-difference matrix with $(C_k)_{ii}=1$ for $0\leq i\leq k$ and $(C_k)_{i,i+1}=-1$ for $0\leq i<k$. It is invertible.  Write $G_k=C_kH_kC_k^{\mathsf T}$.  For $i,j<k$,
first differencing in the row and column indices gives
\[
 (G_k)_{ij}=(H_k)_{ij}-(H_k)_{i+1,j}-(H_k)_{i,j+1}
             +(H_k)_{i+1,j+1}.
\]
Since the entries of $H_k$ depend only on $i+j$, the right-hand side vanishes
except when $i+j$ is one of $k-1,k,k+1$.  Substituting
\eqref{eq:H-entries} gives, up to symmetry,
\begin{equation}\label{eq:G-bands}
 (G_k)_{ij}=\begin{cases}
  -x/(x-1),&i,j<k,\ i+j=k-1,\\[1mm]
  (x+1)/(x-1),&i,j<k,\ i+j=k,\\[1mm]
  -1/(x-1),&i,j<k,\ i+j=k+1.
 \end{cases}
\end{equation}
For $k\geq2$, the last column satisfies $(G_k)_{i,k}=(H_k)_{i,k}-(H_k)_{i+1,k}$ when $i<k$, and hence $(G_k)_{0k}=x/(x-1)$, $(G_k)_{1k}=-1/(x-1)$, and $(G_k)_{kk}=0$. On the other hand, $C_kB_kC_k^{\mathsf T}$ is tridiagonal, with
\begin{equation}\label{eq:CBCT}
 (C_kB_kC_k^{\mathsf T})_{ii}=
 \begin{cases}b_i+b_{i+1},&i<k,\\ b_k,&i=k,\end{cases}
 \qquad
 (C_kB_kC_k^{\mathsf T})_{i,i+1}=-b_{i+1}.
\end{equation}
Combining \eqref{eq:HT-congruence} and
\eqref{eq:BDelta-congruence}, for $\eps\in\{1,-1\}$ we have
\begin{equation}\label{eq:combined-congruence}
W_k^{-1}P_k(\Delta_k+\eps T_k)P_kW_k^{-1}
=
B_k+\eps H_k.
\end{equation}
Since $W_k$ is diagonal and $P_k^{\mathsf T}=P_k$, setting
$S=P_kW_k^{-1}$ rewrites the left-hand side of
\eqref{eq:combined-congruence} as
$S^{\mathsf T}(\Delta_k+\eps T_k)S$.  As $S$ is invertible,
$\Delta_k+\eps T_k$ is congruent to $B_k+\eps H_k$.
Since $C_k$ is invertible, the latter is in turn congruent to
\begin{equation}\label{eq:common-congruent-matrix}
 \mathcal Q_k^\eps=C_k(B_k+\eps H_k)C_k^{\mathsf T}.
\end{equation}

\subsection{Endpoint pairing and tent polynomials}

For the block analysis below assume $k\geq2$; the case $k=1$ is treated
directly in the odd-endpoint argument below.  Set $r=\lfloor k/2\rfloor$.

Put $L=x-1$, $\alpha=x/L$, and $\beta=1/L$.  Permute the rows and columns
of \(\mathcal Q_k^\varepsilon\) simultaneously according to the opposite-endpoint
ordering
\[
 (0,k),(1,k-1),(2,k-2),\ldots.
\]
We refer to the pair $(h,k-h)$ in this ordering as pair $h$. Within each pair we keep the order $(h,k-h)$, and let $D_h^\eps$ denote the $2\times2$ principal block of $\mathcal Q_k^\varepsilon$ on rows and columns $h$ and $k-h$. If $k=2r+1$, this produces $r+1$ pairs; if $k=2r$, it produces $r$ pairs
and the singleton $r$. We call the pairs indexed by $0\leq h<r$ the nonterminal pairs. From \eqref{eq:G-bands}, \eqref{eq:CBCT}, and \eqref{eq:common-congruent-matrix}, the first
pair block and the subsequent nonterminal pair blocks are
\begin{equation}\label{eq:D0-block}
 D_0^\eps=\begin{pmatrix}b_0+b_1&\eps\alpha\\
                           \eps\alpha&b_k\end{pmatrix},
\end{equation}
and, for $1\leq h<r$,
\begin{equation}\label{eq:Dh-block}
 D_h^\eps=
 \begin{pmatrix}
  b_h+b_{h+1}&\eps(\alpha+\beta)\\
  \eps(\alpha+\beta)&b_{k-h}+b_{k-h+1}
 \end{pmatrix}.
\end{equation}
Whenever the pair $h+1$ exists, let $E_h^\eps$ denote the $2\times2$ block of $\mathcal Q_k^\varepsilon$ with rows indexed by $(h,k-h)$ and columns indexed by $(h+1,k-h-1)$. By \eqref{eq:G-bands}, \eqref{eq:CBCT}, and \eqref{eq:common-congruent-matrix},
\begin{equation}\label{eq:Eh-block}
 E_h^\eps=
 \begin{pmatrix}
  -b_{h+1}&-\eps\alpha\\
  -\eps\beta&-b_{k-h}
 \end{pmatrix}.
\end{equation}
All other blocks between nonterminal pairs vanish. Thus, in the paired
ordering, the matrix is block tridiagonal: the nonterminal diagonal blocks are
$D_h^\eps$, the successive couplings are $E_h^\eps$, and the final odd pair
or even singleton is coupled only to the preceding pair. The entries of the
final block and its coupling are recorded explicitly in Appendix~\ref{app:endpoints}.

The following polynomials will encode the determinants arising in the successive Schur complements below. For $1\leq h\leq\lceil k/2\rceil$ define the truncated tent polynomial
\begin{equation}\label{eq:tent-definition}
 A_{k,h}(x)=x^k+4\sum_{\ell=0}^{k-1}(-1)^{\ell+1}
 \min\{\ell+1,k-\ell,h\}x^{k-1-\ell}.
\end{equation}
Increasing the truncation height from \(h\) to \(h+1\)
changes the minimum term by \(1\) precisely when
\(
h\leq \ell\leq k-h-1,
\)
and leaves it unchanged otherwise. Hence
\begin{equation}
A_{k,h+1}(x)-A_{k,h}(x)
=
4\sum_{\ell=h}^{k-h-1}
(-1)^{\ell+1}x^{k-1-\ell}.
\end{equation}
Factoring out \(4(-1)^{h+1}x^h\) leaves the alternating sum
\(R_{k-1-2h}(x)\). Therefore,
\begin{equation}
A_{k,h+1}(x)-A_{k,h}(x)
=
4(-1)^{h+1}x^hR_{k-1-2h}(x)
\quad
\left(1\leq h<\left\lceil\frac{k}{2}\right\rceil\right).
\label{eq:tent-recurrence}
\end{equation}

\subsection{Positivity of the tent polynomials}

\begin{lemma}\label{lem:tent-positive}
The following statements hold.
\begin{enumerate}
 \item If $k=2r+1$ and $x\geq5$, then
 $A_{k,h}(x)>0$ for $1\leq h\leq r+1$.
 \item If $k=2r$ and $x\geq3$, then
 $A_{k,h}(x)>0$ for $1\leq h\leq r$.
\end{enumerate}
\end{lemma}

\begin{proof}
For odd $k=2r+1$, put
$u_\ell=\min\{\ell+1,k-\ell,h\}$.  Pairing consecutive terms in
\eqref{eq:tent-definition} gives
\begin{equation}\label{eq:odd-tent-decomposition}
 A_{k,h}=x^{k-1}(x-4)
 +4\sum_{j=1}^{r}x^{k-2j-1}\bigl(xu_{2j-1}-u_{2j}\bigr).
\end{equation}
Here $u_{2j}\leq u_{2j-1}+1\leq2u_{2j-1}$.  Every summand in
\eqref{eq:odd-tent-decomposition}, as well as the initial term, is therefore
positive for $x\geq5$.

Now let $k=2r$. When $h=1$, \eqref{eq:tent-definition} and
\eqref{eq:R-definition} give
\begin{equation}\label{eq:even-tent-height-one}
\begin{aligned}
 A_{2r,1}(x)
 &=x^{2r}+4\sum_{\ell=0}^{2r-1}(-1)^{\ell+1}x^{2r-1-\ell}\\
 &=x^{2r}-4R_{2r-1}(x)\\
 &=\frac{x^{2r}(x-3)+4}{x+1}>0\qquad(x\geq3).
\end{aligned}
\end{equation}
Write $\Gamma_h=A_{k,h+1}-A_{k,h}$. By \eqref{eq:tent-recurrence}, an odd-indexed increment is
$\Gamma_{2j-1}=4x^{2j-1}R_{k-4j+1}>0$.
When the following even-indexed increment also exists, put
$\nu=k-4j+1$, which is odd.  By \eqref{eq:R-definition},
$R_\nu=x^\nu-x^{\nu-1}+R_{\nu-2}$ and
$R_{\nu-2}=(x^{\nu-1}-1)/(x+1)$. Hence
\begin{align*}
 \Gamma_{2j-1}+\Gamma_{2j}
 &=4x^{2j-1}\bigl(R_\nu-xR_{\nu-2}\bigr)\\
 &=4x^{2j-1}(x-1)\bigl(x^{\nu-1}-R_{\nu-2}\bigr)\\
 &=4x^{2j-1}(x-1)\frac{x^{\nu}+1}{x+1}>0.
\end{align*}
Thus $A_{k,h}-A_{k,1}$ is a sum of positive adjacent pairs, possibly
followed by one positive odd-indexed increment.  This proves the even claim.
\end{proof}

\subsection{The nonterminal Schur recurrence}

In the parameter ranges of Theorem~\ref{thm:parity},
\eqref{eq:R-definition} and \eqref{eq:F-definition} give
\[
 R_i(x)>0,\qquad F_i(x)>0,
\]
and Lemma~\ref{lem:tent-positive} gives
\[
 A_{k,h+1}(x)>0\qquad(0\leq h<r).
\]
Define, for every nonterminal pair $0\leq h<r$,
\begin{equation}\label{eq:cga-definition}
 c_h=\frac{2R_h}{x^{k-h-1}L},\qquad
 g_h=\frac{F_{h+1}A_{k,h+1}}{x^{k-1}L^2F_h},\qquad
 a_h=\frac{g_h+\alpha^2}{c_h}.
\end{equation}
These quantities satisfy $c_h>0$ and $g_h>0$.

\begin{lemma}
\label{lem:universal-schur}
After successive block elimination from the left, the Schur complement at
pair $h$, for $0\leq h<r$, is
\begin{equation}\label{eq:universal-Schur}
 S_h^\eps=
 \begin{pmatrix}a_h&\eps\alpha\\ \eps\alpha&c_h\end{pmatrix},
 \qquad \det S_h^\eps=g_h.
\end{equation}
In particular, since $c_h>0$ and $g_h>0$, each $S_h^\eps$ is positive
definite and hence invertible.
\end{lemma}

\begin{proof}
We prove by induction on $h$ that the successive Schur complements have the form stated in \eqref{eq:universal-Schur}. For $h=0$, \eqref{eq:D0-block} has lower-right entry
$b_k=2/(x^{k-1}L)=c_0$ and off-diagonal entry $\eps\alpha$.
Moreover, \eqref{eq:b-entries} gives
\[
 \det D_0^\eps
 =(b_0+b_1)b_k-\alpha^2
 =\frac{(x-2)\bigl(x^k-4R_{k-1}\bigr)}{x^{k-1}L^2}.
\]
By \eqref{eq:F-definition}, $F_0=1$ and $F_1=x-2$, while \eqref{eq:tent-definition} together with \eqref{eq:R-definition} gives $A_{k,1}=x^k-4R_{k-1}$. Hence the last expression equals $g_0$ by \eqref{eq:cga-definition}, and the upper-left entry of $D_0^\eps$ is therefore $a_0$.

Suppose \eqref{eq:universal-Schur} holds at $h-1$, where $h\geq1$.  Set
$c=c_{h-1}$, $g=g_{h-1}$, $a=a_{h-1}$, and $u=b_h$.  Since
$g=g_{h-1}>0$, the inductive block $S_{h-1}^\eps$ is invertible. Taking the inverse in \eqref{eq:universal-Schur} and using \eqref{eq:Eh-block}, with $b_{k-h+1}=c$ from \eqref{eq:b-entries}, gives
\begin{equation}\label{eq:generic-inverse-coupling}
 (S_{h-1}^\eps)^{-1}=\frac1g
 \begin{pmatrix}c&-\eps\alpha\\-\eps\alpha&a\end{pmatrix},
 \qquad
 E_{h-1}^\eps=
 \begin{pmatrix}-u&-\eps\alpha\\-\eps\beta&-c\end{pmatrix}.
\end{equation}

By block tridiagonality, eliminating the pairs preceding $h-1$ leaves
$D_h^\eps$ and $E_{h-1}^\eps$ unchanged. Thus, at the $h$-th elimination
step, the relevant $2\times2$ block matrix is
\[
 \begin{pmatrix}
  S_{h-1}^\eps&E_{h-1}^\eps\\
  (E_{h-1}^\eps)^{\mathsf T}&D_h^\eps
 \end{pmatrix}.
\]
Eliminating the preceding block therefore gives
\begin{equation}\label{eq:generic-Schur-step}
 S_h^\eps
 =D_h^\eps-(E_{h-1}^\eps)^{\mathsf T}
  (S_{h-1}^\eps)^{-1}E_{h-1}^\eps.
\end{equation}

The second column of $E_{h-1}^\eps$ has the particularly simple image
\begin{equation}\label{eq:generic-column-cancellation}
 (S_{h-1}^\eps)^{-1}
 \begin{pmatrix}-\eps\alpha\\-c\end{pmatrix}
 =\begin{pmatrix}0\\-1\end{pmatrix},
\end{equation}
because $ac=g+\alpha^2$.  Consequently, the correction term in
\eqref{eq:generic-Schur-step} satisfies
\[
 \bigl[(E_{h-1}^\eps)^{\mathsf T}(S_{h-1}^\eps)^{-1}
 E_{h-1}^\eps\bigr]_{22}=c,
 \qquad
 \bigl[(E_{h-1}^\eps)^{\mathsf T}(S_{h-1}^\eps)^{-1}
 E_{h-1}^\eps\bigr]_{12}=\eps\beta.
\]
Subtracting these entries from \eqref{eq:Dh-block} yields
\[
 (S_h^\eps)_{22}=b_{k-h}=c_h,
 \qquad
 (S_h^\eps)_{12}=\eps(\alpha+\beta)-\eps\beta=\eps\alpha.
\]

For the determinant, the first column of $E_{h-1}^\eps$ gives the
upper-left correction $\frac{cu^2-2u\alpha\beta+a\beta^2}{g}$.  Since
$ac=g+\alpha^2$, this can be rewritten as
$\frac{\beta^2}{c}+\frac{(cu-\alpha\beta)^2}{cg}$.  Thus the upper-left
entry is $u+b_{h+1}$ minus this correction, and therefore
\begin{equation}\label{eq:Schur-determinant-formula}
 \det S_h^\eps
 =c_h\left(
 u+b_{h+1}-\frac{\beta^2}{c}
 -\frac{(cu-\alpha\beta)^2}{cg}
 \right)-\alpha^2.
\end{equation}

For use in both the nonterminal and endpoint calculations, we record the
following polynomial identity, valid for every
$1\leq h\leq\lceil k/2\rceil$. Expanding $R_{h-1}R_{k-h}$ using \eqref{eq:R-definition}, we find that
the coefficient of $x^{k-1-\ell}$ is
$(-1)^\ell\min\{\ell+1,k-\ell,h\}$ for $0\leq\ell\leq k-1$.
Comparing with \eqref{eq:tent-definition} therefore gives
\begin{equation}\label{eq:tent-product-identity}
 A_{k,h}=x^k-4R_{h-1}R_{k-h}
 \qquad\left(1\leq h\leq\left\lceil\frac{k}{2}\right\rceil\right).
\end{equation}
Using \eqref{eq:b-entries}, \eqref{eq:cga-definition}, and
\eqref{eq:tent-product-identity}, together with the local definitions of
$c$, $g$, and $u$, we obtain
\begin{equation}\label{eq:cu-relations}
 cu-\alpha\beta
 =-\frac{A_{k,h}}{x^{k-1}L^2},
 \qquad
 \frac{cu-\alpha\beta}{g}
 =-\frac{F_{h-1}}{F_h}.
\end{equation}
Substituting \eqref{eq:cu-relations} into
\eqref{eq:Schur-determinant-formula} and carrying out the algebraic
simplification detailed in Appendix~\ref{app:schur-algebra}, we obtain
\[
x^{k-1}L^2F_h\det S_h^\eps
=
F_{h+1}\bigl(x^k-4R_hR_{k-h-1}\bigr)
=
F_{h+1}A_{k,h+1},
\]
where the last equality is \eqref{eq:tent-product-identity} with
$h+1$ in place of $h$. Hence
\begin{equation}
\det S_h^\eps
=
\frac{F_{h+1}A_{k,h+1}}
{x^{k-1}L^2F_h}
=
g_h.
\end{equation}
The upper-left entry is now uniquely determined by the lower-right entry,
the off-diagonal entry, and the determinant, and equals $a_h$.
This completes the induction.
\end{proof}

\begin{remark}
Formula \eqref{eq:even-tent-height-one} shows that, for fixed $1<x<3$,
$A_{2r,1}(x)$ is negative for all sufficiently large $r$.  Thus $x=3$ is
the natural threshold for the uniform tent-positivity argument used here.  We
do not infer from this alone that the first block determinant is negative: that
determinant also contains $F_1=x-2$, which changes sign.
\end{remark}

\subsection{Odd and even endpoints}

We use the endpoint identities collected in Appendix~\ref{app:endpoints}.

We now complete the proof of Theorem~\ref{thm:parity}.  In the odd case
$k=2r+1$ we have $x\geq5$, whereas in the even case $k=2r$ we have
$x\geq3$.  By Lemma~\ref{lem:universal-schur}, every nonterminal Schur block is
positive definite and hence invertible.

Suppose first that $k=2r+1$ is odd.  For $r=0$,
\begin{equation}
 T_1=
 \begin{pmatrix}
  -(x-1)&x-1\\
  x-1&1
 \end{pmatrix},
 \qquad
 \Delta_1=\operatorname{diag}(2(x-1),x-2).
\end{equation}
Thus
\begin{equation}
 \Delta_1+T_1
 =(x-1)
 \begin{pmatrix}
  1&1\\
  1&1
 \end{pmatrix}.
\end{equation}
This matrix is positive semidefinite of rank one.  By
\eqref{eq:parity-identity} with $k=1$, $T_1s_1=-\Delta_1s_1$, and hence
$(\Delta_1+T_1)s_1=0$.  Since its nullity is one,
$\ker(\Delta_1+T_1)=\Span\{s_1\}$.  Moreover,
$\Delta_1-T_1$ has leading diagonal entry $3(x-1)>0$ and determinant
$2(x-1)(x-4)>0$, and is therefore positive definite.

Now let $r\geq1$.  All coordinates are paired.  Since $k$ is odd,
\eqref{eq:parity-identity} gives $(\Delta_k+T_k)s_k=0$, so
\eqref{eq:combined-congruence}--\eqref{eq:common-congruent-matrix} show
that $\mathcal Q_k^+$ is singular.  Since $g_h>0$ for all nonterminal
blocks, the determinant factorisation \eqref{eq:appendix-odd-det-factorisation} forces
$\det S_r^+=0$.  Together with \eqref{eq:appendix-odd-plus-entries}, this
determines the final plus block as
\begin{equation}\label{eq:odd-final-plus}
 S_r^+
 =
 \frac1{x^r(x-1)F_r}
 \begin{pmatrix}
  F_{r+1}^2&F_{r+1}F_r\\
  F_{r+1}F_r&F_r^2
 \end{pmatrix}.
\end{equation}
The scalar factor is positive, and
\begin{equation}
 \begin{pmatrix}
  F_{r+1}^2&F_{r+1}F_r\\
  F_{r+1}F_r&F_r^2
 \end{pmatrix}
 =
 \begin{pmatrix}
  F_{r+1}\\
  F_r
 \end{pmatrix}
 \begin{pmatrix}
  F_{r+1}&F_r
 \end{pmatrix}.
\end{equation}
Hence $S_r^+$ is positive semidefinite.  Since $F_r>0$, it is nonzero
and therefore has rank one.

For the minus sign, put
\begin{equation}\label{eq:odd-final-parameters}
 C_r^*=\frac{x^r+2R_r}{x^r(x-1)},\qquad
 B_r^*=-\frac{x^{r+1}+2R_r}{x^r(x-1)},\qquad
 G_r^*=\frac{2A_{k,r+1}}{x^r(x-1)F_r}.
\end{equation}
By \eqref{eq:appendix-odd-minus-entries} and
\eqref{eq:appendix-odd-minus-det-final}, the final Schur block is
\begin{equation}\label{eq:odd-final-minus}
 S_r^-=
 \begin{pmatrix}
  (G_r^*+(B_r^*)^2)/C_r^*&B_r^*\\
  B_r^*&C_r^*
 \end{pmatrix},
 \qquad
 \det S_r^-=G_r^*.
\end{equation}
Since $R_r>0$, we have $C_r^*>0$.  Moreover,
Lemma~\ref{lem:tent-positive} gives $A_{k,r+1}>0$ in the odd case,
and $F_r>0$, so $G_r^*>0$.  Hence
\eqref{eq:odd-final-minus} gives $S_r^-\succ0$.

We now transfer these endpoint conclusions back to the original matrices.
By \eqref{eq:combined-congruence} and
\eqref{eq:common-congruent-matrix},
$\Delta_k+\eps T_k$ is congruent to $\mathcal Q_k^\eps$.
The simultaneous row-and-column permutation used in the endpoint pairing
is also a congruence transformation. Moreover, for a real symmetric block
matrix with invertible leading block $S$,
\begin{equation}\label{eq:schur-congruence-step}
 \begin{pmatrix}
  S&E\\
  E^{\mathsf T}&D
 \end{pmatrix}
 \quad\text{is congruent to}\quad
 S\oplus\bigl(D-E^{\mathsf T}S^{-1}E\bigr).
\end{equation}
Since the nonterminal blocks $S_0^\eps,\ldots,S_{r-1}^\eps$ are positive
definite, each Schur elimination in \eqref{eq:schur-congruence-step} is
legitimate.  Consequently, after the successive eliminations in the odd
case,
\begin{equation}\label{eq:odd-schur-direct-sum}
 \mathcal Q_k^\eps
 \quad\text{is congruent to}\quad
 S_0^\eps\oplus S_1^\eps\oplus\cdots\oplus S_r^\eps.
\end{equation}
Since congruence preserves inertia, combining
\eqref{eq:odd-schur-direct-sum} with the positivity of the nonterminal
blocks established above shows that $\Delta_k+T_k$ is positive semidefinite.  The nonterminal blocks contribute no zero eigenvalues, while the $2\times2$ block $S_r^+$ is nonzero and has rank one; hence the nullity is exactly one.  The positive definiteness of $S_r^-$ similarly gives
$\Delta_k-T_k\succ0$.  Finally, since $k$ is odd,
\eqref{eq:parity-identity} gives $(\Delta_k+T_k)s_k=0$.  The nullity is
one, and therefore $\ker(\Delta_k+T_k)=\Span\{s_k\}$.

Now let $k=2r$ be even.  Abbreviate
\begin{equation}\label{eq:even-end-abbreviations}
 c=c_{r-1},\qquad
 g=g_{r-1},\qquad
 B=b_r,\qquad
 \alpha=\frac{x}{x-1},\qquad
 \beta=\frac1{x-1}.
\end{equation}
By \eqref{eq:universal-Schur} and
\eqref{eq:appendix-even-endpoint}, with $b_{r+1}=c$ from
\eqref{eq:b-entries} and \eqref{eq:cga-definition}, the final pair block, the $2\times1$ off-diagonal
block between the final pair and the central singleton, and the $1\times1$
central diagonal block are respectively
\begin{equation}\label{eq:even-final-data}
 S_{r-1}^\eps=
 \begin{pmatrix}
  (g+\alpha^2)/c&\eps\alpha\\
  \eps\alpha&c
 \end{pmatrix},
 \qquad
 v_\eps=
 \begin{pmatrix}
  -B-\eps\alpha\\
  -c-\eps\beta
 \end{pmatrix},
 \qquad
 \eta_\eps=B+c+\eps(\alpha+\beta).
\end{equation}
Since $S_{r-1}^\eps$ is a nonterminal positive-definite block, it is
invertible. Eliminating this block leaves the scalar Schur complement
\begin{equation}\label{eq:even-final-schur-scalar}
 \sigma_\eps
 =
 \eta_\eps
 -
 v_\eps^{\mathsf T}(S_{r-1}^\eps)^{-1}v_\eps.
\end{equation}
Using \eqref{eq:appendix-even-quadratic} in
\eqref{eq:even-final-schur-scalar} gives
\begin{equation}\label{eq:even-sigma}
 \sigma_\eps
 =B-\frac{\beta^2}{c}
 -\frac{(cB-\alpha\beta)^2}{cg}
 +\eps,
\end{equation}
where we used $\alpha-\beta=1$.

By the same congruence and Schur elimination argument as above,
\begin{equation}\label{eq:even-schur-direct-sum}
 \Delta_k+\eps T_k
 \quad\text{is congruent to}\quad
 S_0^\eps\oplus\cdots\oplus
 S_{r-1}^\eps\oplus[\sigma_\eps].
\end{equation}
Since $k$ is even, \eqref{eq:parity-identity} gives
$T_ks_k=\Delta_ks_k$, and hence $(\Delta_k-T_k)s_k=0$.
Thus $\Delta_k-T_k$ is singular.  For $\eps=-1$, the blocks
$S_0^-,\ldots,S_{r-1}^-$ are positive definite and hence nonsingular,
so \eqref{eq:even-schur-direct-sum} forces $\sigma_-=0$.
Since \eqref{eq:even-sigma} gives $\sigma_+-\sigma_-=2$, we obtain $\sigma_-=0$ and $\sigma_+=2$. Since all preceding Schur blocks are positive definite,
\eqref{eq:even-schur-direct-sum} shows that $\Delta_k-T_k$ is positive
semidefinite with nullity one, whereas $\Delta_k+T_k$ is positive
definite.  Finally, \eqref{eq:parity-identity} and the one-dimensional
kernel give $\ker(\Delta_k-T_k)=\Span\{s_k\}$.

This proves both halves of Theorem~\ref{thm:parity}.

\section{Tensor semidefinite bounds and energy maximisation}
\label{sec:tensor-bounds}

\subsection{A sharp tensor \texorpdfstring{$\ell_1$}{l1}-bound}

Assume the hypotheses of Theorem~\ref{thm:main}, and retain the notation of
\eqref{eq:index-set}--\eqref{eq:tensor-operator}. Let $X$ be the admissible
divisor tensor encoding $\mathcal D$ as in \eqref{eq:X-definition}.
Define the associated sign vector by centring:
\begin{equation}\label{eq:center-X}
 \mathbf y=2\mathbf x-\one\in\{\pm1\}^{N},\qquad
 \mathbf x=\frac12(\one+\mathbf y).
\end{equation}
This centring separates the contribution $K\one$, determined by the
row-sum identity, from $K\mathbf y$, whose $\ell_1$-norm will be bounded
using the tensor consequences of Theorem~\ref{thm:parity}.  We continue to use the
finite-factor Kronecker identities collected in Appendix~\ref{app:finite-kron}.
Put
\begin{equation}\label{eq:factor-Deltas}
 \Delta_\nu=\Delta_{a_\nu}(p_\nu)\quad(1\leq\nu\leq m),\qquad
 \Delta=\Delta_m\otimes\cdots\otimes\Delta_1,
\end{equation}
and define
\begin{equation}\label{eq:tensor-parity-vector}
 \mathbf s=s_{a_m}\otimes\cdots\otimes s_{a_1},\qquad
 s_{\mathbf i}=(-1)^{i_1+\cdots+i_m}\quad(\mathbf i\in\mathcal I).
\end{equation}
Since $s_{a_\nu}\in\mathbb R^{a_\nu+1}$ and
$(s_{a_\nu})_{i_\nu}=(-1)^{i_\nu}$, and the multi-index ordering is that of
\eqref{eq:vec-convention}, \eqref{eq:finite-coordinate} gives the second
formula in \eqref{eq:tensor-parity-vector}. Each $\Delta_\nu$ is positive
diagonal by \eqref{eq:delta-definition}--\eqref{eq:delta-last} and the prime
restrictions in Theorem~\ref{thm:main}; hence $\Delta$ is positive diagonal
by \eqref{eq:finite-diagonal-power}.

Write $\Energy_X=\Energy(\ICG_n(\mathcal D))$, and recall that
$e_{\mathbf a}=e_{a_m}\otimes\cdots\otimes e_{a_1}$ is the standard basis
vector at the top coordinate.  By \eqref{eq:T-symmetric-rowsum},
$M_\nu\one_{a_\nu+1}=p_\nu^{a_\nu}e_{a_\nu}$ for every $\nu$. Since
$\one_N=\one_{a_m+1}\otimes\cdots\otimes\one_{a_1+1}$,
\eqref{eq:finite-action} and \eqref{eq:finite-scalars} give
\begin{equation}\label{eq:mean-term}
\begin{aligned}
 K\one_N
 &=(M_m\one_{a_m+1})\otimes\cdots\otimes(M_1\one_{a_1+1})\\
 &=\left(\prod_{\nu=1}^m p_\nu^{a_\nu}\right)
   (e_{a_m}\otimes\cdots\otimes e_{a_1})\\
 &=n e_{\mathbf a}.
\end{aligned}
\end{equation}
Combining \eqref{eq:center-X} and \eqref{eq:mean-term} with the energy
formula \eqref{eq:exact-matrix-energy} from
Proposition~\ref{prop:energy-model} gives
\begin{equation}\label{eq:outer-triangle}
 2\Energy_X
 =2\lVert K\mathbf x\rVert_1
 =\lVert K(\one+\mathbf y)\rVert_1
 =\lVert n e_{\mathbf a}+K\mathbf y\rVert_1
 \leq n+\lVert K\mathbf y\rVert_1.
\end{equation}

\begin{proposition}\label{prop:centered-tensor}
For every $\mathbf y\in\{\pm1\}^{N}$,
\begin{equation}\label{eq:centered-tensor-bound}
 \lVert K\mathbf y\rVert_1\leq
 \prod_{\nu=1}^{m}d_{a_\nu}(p_\nu).
\end{equation}
Equality holds if and only if
\begin{equation}\label{eq:centered-equality}
 \mathbf y=\pm\mathbf s.
\end{equation}
\end{proposition}

\begin{proof}
We first prove that $\Delta+K$ is positive semidefinite with kernel
$\Span\{\mathbf s\}$, and that $\Delta-K$ is positive definite. For
$1\leq\nu\leq m$, define
\begin{equation}\label{eq:factor-normalisation}
 U_\nu=\Delta_\nu^{-1/2}M_\nu\Delta_\nu^{-1/2},\qquad
 \eps_\nu=(-1)^{a_\nu}.
\end{equation}
By \eqref{eq:T-symmetric-rowsum}, $M_\nu$ is real symmetric.
Since $\Delta_\nu$ is positive diagonal, it is real symmetric, and
$\Delta_\nu^{-1/2}$ is also real symmetric and invertible. Hence each
$U_\nu$ is real symmetric. The prime restrictions in
Theorem~\ref{thm:main} ensure that Theorem~\ref{thm:parity} applies to
$M_\nu=T_{a_\nu}(p_\nu)$. Here $I$ denotes the identity matrix of the
appropriate order. Thus
\[
 I\mp\eps_\nu U_\nu
 =\Delta_\nu^{-1/2}(\Delta_\nu\mp\eps_\nu M_\nu)\Delta_\nu^{-1/2},
\]
so each matrix on the left is obtained from the corresponding matrix in
parentheses by an invertible congruence. By Theorem~\ref{thm:parity} and
invariance of positive semidefiniteness and positive definiteness under
invertible congruence, we obtain $I-\eps_\nu U_\nu\succeq0$ and
$I+\eps_\nu U_\nu\succ0$.

To identify the kernel, let $x\in\mathbb R^{a_\nu+1}$. Then
\[
\begin{aligned}
 x\in\ker(I-\eps_\nu U_\nu)
 &\iff \Delta_\nu^{-1/2}(\Delta_\nu-\eps_\nu M_\nu)
          \Delta_\nu^{-1/2}x=0\\
 &\iff (\Delta_\nu-\eps_\nu M_\nu)\Delta_\nu^{-1/2}x=0
 \iff \Delta_\nu^{-1/2}x\in\ker(\Delta_\nu-\eps_\nu M_\nu).
\end{aligned}
\]
By Theorem~\ref{thm:parity}, the last kernel is $\Span\{s_{a_\nu}\}$, so
\[
 \ker(I-\eps_\nu U_\nu)=\Span\{\Delta_\nu^{1/2}s_{a_\nu}\}.
\]
Since $\Delta_\nu^{1/2}s_{a_\nu}\neq0$, the kernel above is nontrivial. Moreover, since $\eps_\nu^2=1$,
\[
 x\in\ker(I-\eps_\nu U_\nu)\iff U_\nu x=\eps_\nu x.
\]
Hence $\eps_\nu$ is an eigenvalue of $U_\nu$ with one-dimensional eigenspace $\Span\{\Delta_\nu^{1/2}s_{a_\nu}\}$.

Let $\lambda$ be any eigenvalue of $U_\nu$ with eigenvector $v\neq0$, so that $U_\nu v=\lambda v$. Since $I-\eps_\nu U_\nu\succeq0$,
\[
 0\leq v^{\mathsf T}(I-\eps_\nu U_\nu)v
 =(1-\eps_\nu\lambda)\lVert v\rVert_2^2,
\]
so $\eps_\nu\lambda\leq1$. Similarly, $I+\eps_\nu U_\nu\succ0$ gives $(1+\eps_\nu\lambda)\lVert v\rVert_2^2>0$, hence $\eps_\nu\lambda>-1$. Therefore $-1<\eps_\nu\lambda\leq1$, so $|\lambda|\leq1$. If $|\lambda|=1$, then $\lambda=\pm\eps_\nu$. If $\lambda=-\eps_\nu$, then $\eps_\nu\lambda=\eps_\nu(-\eps_\nu)=-\eps_\nu^2=-1$, contradicting $\eps_\nu\lambda>-1$. Thus $\eps_\nu$ is the unique eigenvalue of $U_\nu$ of modulus $1$.

Now set $U=U_m\otimes\cdots\otimes U_1\in\R^{N\times N}$, where the dimension follows from \eqref{eq:finite-dimension}. Since each $\Delta_\nu$ is positive diagonal, \eqref{eq:finite-diagonal-power} with $\alpha=-1/2$ gives
\[
\Delta^{-1/2}=\Delta_m^{-1/2}\otimes\cdots\otimes\Delta_1^{-1/2}.
\]
Hence, by \eqref{eq:finite-mixed-product}, \eqref{eq:tensor-operator}, and \eqref{eq:factor-normalisation},
\begin{equation}
\begin{aligned}
\Delta^{-1/2}K\Delta^{-1/2}
&=(\Delta_m^{-1/2}\otimes\cdots\otimes\Delta_1^{-1/2})
  (M_m\otimes\cdots\otimes M_1)
  (\Delta_m^{-1/2}\otimes\cdots\otimes\Delta_1^{-1/2})\\
&=(\Delta_m^{-1/2}M_m\Delta_m^{-1/2})\otimes\cdots\otimes(\Delta_1^{-1/2}M_1\Delta_1^{-1/2})\\
&=U_m\otimes\cdots\otimes U_1=U.
\end{aligned}
\end{equation}

By the spectral theorem for real symmetric matrices, for each $\nu$ choose an
orthonormal eigenbasis $\{v_{\nu,j}:0\leq j\leq a_\nu\}$ of
$\mathbb R^{a_\nu+1}$, with $U_\nu v_{\nu,j}=\lambda_{\nu,j}v_{\nu,j}$.
For $\mathbf j=(j_1,\ldots,j_m)\in\mathcal I$, put
$V_{\mathbf j}=v_{m,j_m}\otimes\cdots\otimes v_{1,j_1}\in\R^N$,
where the dimension follows from \eqref{eq:finite-dimension}. Since the
$v_{\nu,j}$ form orthonormal bases, for $\mathbf j,\mathbf k\in\mathcal I$,
\eqref{eq:finite-inner-product} gives
\[
 \langle V_{\mathbf j},V_{\mathbf k}\rangle
 =\prod_{\nu=1}^{m}\langle v_{\nu,j_\nu},v_{\nu,k_\nu}\rangle
 =\prod_{\nu=1}^{m}\delta_{j_\nu k_\nu},
\]
where $\delta_{rs}$ denotes the Kronecker delta, equal to $1$ if $r=s$ and
$0$ otherwise. Thus the $N=\prod_{\nu=1}^{m}(a_\nu+1)$ vectors
$V_{\mathbf j}$ form an orthonormal basis of $\mathbb R^N$. By
\eqref{eq:finite-action} and \eqref{eq:finite-scalars},
\begin{equation}
 U V_{\mathbf j}
 =(U_mv_{m,j_m})\otimes\cdots\otimes(U_1v_{1,j_1})
 =\left(\prod_{\nu=1}^{m}\lambda_{\nu,j_\nu}\right)V_{\mathbf j}.
\end{equation}
Hence the $V_{\mathbf j}$ form an orthonormal eigenbasis of $U$, with
corresponding eigenvalues $\prod_{\nu=1}^{m}\lambda_{\nu,j_\nu}$. For each $\nu$, the unique
eigenvalue of $U_\nu$ of modulus $1$ is $\eps_\nu$, with one-dimensional
eigenspace $\Span\{\Delta_\nu^{1/2}s_{a_\nu}\}$. Let $j_\nu^\ast$ be the unique index for which $\lambda_{\nu,j_\nu^\ast}=\eps_\nu$,
and put $\mathbf j^\ast=(j_1^\ast,\ldots,j_m^\ast)$. Then there is a nonzero scalar $c_\nu$ such that
$v_{\nu,j_\nu^\ast}=c_\nu\Delta_\nu^{1/2}s_{a_\nu}$.

If the eigenvalue $\prod_{\nu=1}^{m}\lambda_{\nu,j_\nu}$ of $U$ has modulus $1$, then
$|\lambda_{\nu,j_\nu}|=1$ for every $\nu$, since each factor has modulus at most $1$.
Hence $\lambda_{\nu,j_\nu}=\eps_\nu$ and therefore $j_\nu=j_\nu^\ast$ for all $\nu$.
Thus exactly one tensor-product basis vector corresponds to an eigenvalue of modulus $1$,
namely $V_{\mathbf j^\ast}=v_{m,j_m^\ast}\otimes\cdots\otimes v_{1,j_1^\ast}$.
By \eqref{eq:finite-scalars},
\[
 V_{\mathbf j^\ast}=\left(\prod_{\nu=1}^{m}c_\nu\right)
 \bigl(\Delta_m^{1/2}s_{a_m}\otimes\cdots\otimes\Delta_1^{1/2}s_{a_1}\bigr).
\]
Since $\prod_{\nu=1}^{m}c_\nu\neq0$, these two vectors span the same one-dimensional
space. Hence the resulting eigenspace of $U$ is
\[
 \Span\{V_{\mathbf j^\ast}\}
 =\Span\{\Delta_m^{1/2}s_{a_m}\otimes\cdots\otimes\Delta_1^{1/2}s_{a_1}\}.
\]
Its eigenvalue is
\begin{equation}\label{eq:tensor-endpoint}
 \prod_{\nu=1}^{m}\eps_\nu=(-1)^{\sum_{\nu=1}^{m}a_\nu}=-1,
\end{equation}
where the last equality uses the hypothesis that
$\sum_{\nu=1}^{m}a_\nu$ is odd. By \eqref{eq:factor-Deltas} and
\eqref{eq:finite-diagonal-power} with $\alpha=1/2$,
\[
 \Delta^{1/2}=\Delta_m^{1/2}\otimes\cdots\otimes\Delta_1^{1/2}.
\]
Hence, by \eqref{eq:tensor-parity-vector} and \eqref{eq:finite-action},
\[
 \Delta^{1/2}\mathbf s
 =(\Delta_m^{1/2}\otimes\cdots\otimes\Delta_1^{1/2})
  (s_{a_m}\otimes\cdots\otimes s_{a_1})
 =\Delta_m^{1/2}s_{a_m}\otimes\cdots\otimes\Delta_1^{1/2}s_{a_1}.
\]
Therefore the eigenspace of $U$ corresponding to the eigenvalue $-1$ is
$\Span\{\Delta^{1/2}\mathbf s\}$.

Since every eigenvalue of $U$ lies in $[-1,1]$ and $-1$ is its unique
eigenvalue of modulus $1$,
\[
 I+U\succeq0,\qquad
 \ker(I+U)=\Span\{\Delta^{1/2}\mathbf s\},\qquad
 I-U\succ0.
\]
Since
$\Delta\pm K=\Delta^{1/2}(I\pm U)\Delta^{1/2}$ and
$\Delta^{1/2}$ is symmetric and invertible, the semidefinite and definite
conclusions transfer by congruence. To identify the kernel, take any
$x\in\mathbb R^N$. Then
\[
 x\in\ker(\Delta+K)\iff (I+U)\Delta^{1/2}x=0
 \iff \Delta^{1/2}x\in\ker(I+U).
\]
Since $\ker(I+U)=\Span\{\Delta^{1/2}\mathbf s\}$, this is equivalent to
$\Delta^{1/2}x\in\Span\{\Delta^{1/2}\mathbf s\}$, and the invertibility of
$\Delta^{1/2}$ gives $x\in\Span\{\mathbf s\}$.  Hence
\begin{equation}\label{eq:kron-kernel}
 \Delta+K\succeq0,\qquad \ker(\Delta+K)=\Span\{\mathbf s\},
\end{equation}
and
\begin{equation}\label{eq:kron-minus-pd}
 \Delta-K\succ0.
\end{equation}

We now convert these tensor semidefinite bounds into the desired
$\ell_1$-estimate.  Choose a sign vector $\mathbf z$ so that
$z_{\mathbf i}=\sgn((K\mathbf y)_{\mathbf i})$ when
$(K\mathbf y)_{\mathbf i}\neq0$, choosing either sign at a zero coordinate.
Then
\begin{equation}\label{eq:z-dual}
 \mathbf z^{\mathsf T}K\mathbf y=\lVert K\mathbf y\rVert_1.
\end{equation}
Since $\Delta$ is diagonal and $\mathbf y,\mathbf z$ are sign vectors,
\begin{equation}\label{eq:Delta-trace-signs}
 \mathbf y^{\mathsf T}\Delta\mathbf y
 =\mathbf z^{\mathsf T}\Delta\mathbf z
 =\tr\Delta=\prod_{\nu=1}^{m}d_{a_\nu}(p_\nu),
\end{equation}
where the final equality follows from \eqref{eq:factor-Deltas},
\eqref{eq:finite-trace}, and \eqref{eq:dk-trace}.
Moreover, by \eqref{eq:T-symmetric-rowsum}, each $M_\nu$ is symmetric.  Hence
\eqref{eq:tensor-operator} and \eqref{eq:finite-transpose} give
\[
 K^{\mathsf T}=M_m^{\mathsf T}\otimes\cdots\otimes M_1^{\mathsf T}=K.
\]
Thus $K$ is symmetric.  Since $\mathbf y^{\mathsf T}K\mathbf z$ is a scalar,
\[
 \mathbf y^{\mathsf T}K\mathbf z
 =(\mathbf y^{\mathsf T}K\mathbf z)^{\mathsf T}
 =\mathbf z^{\mathsf T}K^{\mathsf T}\mathbf y
 =\mathbf z^{\mathsf T}K\mathbf y.
\]
Expanding the quadratic forms and using \eqref{eq:Delta-trace-signs},
we obtain
\begin{align}
 \prod_{\nu=1}^{m}d_{a_\nu}(p_\nu)-\mathbf z^{\mathsf T}K\mathbf y
 =\frac14\bigl[&
 (\mathbf z-\mathbf y)^{\mathsf T}(\Delta+K)(\mathbf z-\mathbf y)\notag\\
 &+(\mathbf z+\mathbf y)^{\mathsf T}(\Delta-K)(\mathbf z+\mathbf y)\bigr].
\label{eq:deficit-identity}
\end{align}
Equations \eqref{eq:kron-kernel}--\eqref{eq:kron-minus-pd} show that the
right-hand side is nonnegative.  Together with \eqref{eq:z-dual}, this
proves \eqref{eq:centered-tensor-bound}.

If equality holds, then both nonnegative quadratic forms in
\eqref{eq:deficit-identity} vanish.  Positive definiteness in
\eqref{eq:kron-minus-pd} forces $\mathbf z+\mathbf y=0$, while
\eqref{eq:kron-kernel} gives
$\mathbf z-\mathbf y\in\Span\{\mathbf s\}$.
Hence $-2\mathbf y=\mathbf z-\mathbf y\in\Span\{\mathbf s\}$.
Since both $\mathbf y$ and $\mathbf s$ are sign vectors, their proportionality
factor is $1$ or $-1$, giving \eqref{eq:centered-equality}.

Conversely, applying \eqref{eq:parity-identity} with
$k=a_\nu$ and $x=p_\nu$ gives
$M_\nu s_{a_\nu}=\eps_\nu\Delta_\nu s_{a_\nu}$ for each $\nu$.
By \eqref{eq:tensor-operator}, \eqref{eq:factor-Deltas},
\eqref{eq:tensor-parity-vector}, \eqref{eq:finite-action}, and
\eqref{eq:finite-scalars}, we obtain
\begin{align}
 K\mathbf s
 &=(M_m\otimes\cdots\otimes M_1)
   (s_{a_m}\otimes\cdots\otimes s_{a_1})\notag\\
 &=(M_m s_{a_m})\otimes\cdots\otimes(M_1s_{a_1})\notag\\
 &=\left(\prod_{\nu=1}^{m}\eps_\nu\right)\Delta\mathbf s
 =-\Delta\mathbf s,
\label{eq:parity-actions}
\end{align}
where the last equality follows from \eqref{eq:tensor-endpoint}.
Since $\Delta$ is positive diagonal and $\mathbf s$ is a sign vector,
\begin{equation}
 \lVert K\mathbf s\rVert_1
 =\lVert\Delta\mathbf s\rVert_1
 =\tr\Delta=\prod_{\nu=1}^{m}d_{a_\nu}(p_\nu).
\end{equation}
The same norm is obtained for $-\mathbf s$.  Thus equality holds for both
choices in \eqref{eq:centered-equality}.
\end{proof}

\subsection{Proof of Theorem~\ref{thm:main}}

\begin{proof}
Combining \eqref{eq:outer-triangle} with
Proposition~\ref{prop:centered-tensor} gives
\begin{equation}\label{eq:energy-upper}
 \Energy_X\leq\frac12\left[n+
 \prod_{\nu=1}^{m}d_{a_\nu}(p_\nu)\right].
\end{equation}
Here $X$ continues to denote an arbitrary admissible divisor tensor as in
\eqref{eq:X-definition}. To show that the upper bound in \eqref{eq:energy-upper}
is attained, we now construct a particular admissible divisor tensor corresponding
to the equality candidate $\mathbf y=\mathbf s$ from Proposition~\ref{prop:centered-tensor}.
In the notation of \eqref{eq:X-definition}, define the checkerboard divisor tensor
$X^*=(x^*_{\mathbf i})_{\mathbf i\in\mathcal I}$ by
\begin{equation}\label{eq:Xstar-general}
 x^*_{\mathbf i}=\frac12\left(1+(-1)^{i_1+\cdots+i_m}\right),\qquad
 \operatorname{vec}(X^*)=\frac12(\one+\mathbf s).
\end{equation}
The vector identity in \eqref{eq:Xstar-general} follows from
\eqref{eq:tensor-parity-vector}.  Thus $x^*_{\mathbf i}=1$ precisely when $\sum_{\nu=1}^{m}i_\nu$ is even.
Moreover, the total exponent is odd, so
\begin{equation}\label{eq:Xstar-corner}
 s_{\mathbf a}=(-1)^{\sum_{\nu=1}^{m}a_\nu}=-1,\qquad
 x^*_{\mathbf a}=0.
\end{equation}
Also $x^*_{(0,\ldots,0)}=1$. Thus $X^*$ is an admissible divisor tensor encoding a nonempty set of
proper divisors, and its associated sign vector $\mathbf y=\mathbf s$ is an
equality case in Proposition~\ref{prop:centered-tensor}. It remains only to verify that equality also
holds in the triangle inequality \eqref{eq:outer-triangle}. The vector
$n e_{\mathbf a}$ is zero except at $\mathbf a$, where its entry is
$n>0$. Since $\Delta$ is positive diagonal, $\Delta_{\mathbf a,\mathbf a}>0$. Hence, by \eqref{eq:parity-actions} and \eqref{eq:Xstar-corner},
\begin{equation}\label{eq:overlap-positive}
 (K\mathbf s)_{\mathbf a}
 =-\Delta_{\mathbf a,\mathbf a}s_{\mathbf a}
 =\Delta_{\mathbf a,\mathbf a}>0.
\end{equation}
Equation \eqref{eq:overlap-positive} rules out cancellation at the only
coordinate where the two vectors are both nonzero.  Equality holds in \eqref{eq:outer-triangle}, and
$X^*$ attains the bound in \eqref{eq:energy-upper}.

If another divisor tensor $X$ encoding a nonempty set of proper divisors
also attains \eqref{eq:energy-upper}, then both inequalities leading to
that bound are equalities.  By \eqref{eq:centered-equality},
$\mathbf y=\pm\mathbf s$.  From \eqref{eq:center-X}, the condition
$x_{\mathbf a}=0$ is equivalent to $y_{\mathbf a}=-1$.
However, \eqref{eq:Xstar-corner} gives $(-\mathbf s)_{\mathbf a}=1$,
so the negative choice is impossible. Hence the only equality case compatible
with admissibility is $\mathbf y=\mathbf s$, which, by \eqref{eq:center-X},
corresponds uniquely to the checkerboard tensor $X^*$. Thus $X^*$ is the
unique admissible divisor tensor attaining the maximum energy. Consequently,
the corresponding divisor set $\mathcal D^*$ is the unique energy-maximising
divisor set. Finally, equality in \eqref{eq:energy-upper} yields the maximal
energy formula \eqref{eq:main-energy}, with the explicit form of $d_k$ given by
\eqref{eq:dk-formula}.
\end{proof}

\begin{remark}
Let $p,q$ be distinct odd primes with $q\geq5$, and let $r\geq1$, $s\geq0$.
Taking $m=2$, $(p_1,p_2)=(p,q)$, and $(a_1,a_2)=(2r,2s+1)$ in
Theorem~\ref{thm:main} gives
\begin{equation}\label{eq:ab-definition}
 n=p^{2r}q^{2s+1},\qquad
 \Dstar=\{p^iq^j:0\leq i\leq2r,\ 0\leq j\leq2s+1,\ i+j\equiv0\pmod2\},
\end{equation}
with maximum energy $\frac12[n+d_{2r}(p)d_{2s+1}(q)]$.
In particular, $r=s=1$ in \eqref{eq:ab-definition} gives
\[
 \{1,p^2,pq,q^2,p^2q^2,pq^3\}
\]
as the unique energy-maximising divisor set, together with Rold\'an's
closed-form energy formula \cite[Theorem~1.2]{roldan2026graph}.
The exceptional prime $q=3$ is not covered by the present theorem.
Indeed, \eqref{eq:tent-definition} and \eqref{eq:R-definition} give
$A_{k,1}(3)=3^k-4R_{k-1}(3)=-1$ for odd $k$, so the uniform
odd-exponent semidefinite argument developed here does not extend to
$x=3$.
\end{remark}

\appendix

\section{Verification of the nonterminal determinant recurrence}
\label{app:schur-algebra}

We record the algebraic reduction used in the proof of
Lemma~\ref{lem:universal-schur}.  Fix $1\leq h<r$, and retain the local
notation $c=c_{h-1}$, $g=g_{h-1}$, and $u=b_h$.  With $L=x-1$, $\alpha=x/L$, and $\beta=1/L$, put
$D=x^{k-1}L^2$, $P=R_{h-1}$, $Q=R_h$, $S=R_{k-h}$, and
$T=R_{k-h-1}$.  By \eqref{eq:b-entries} and
\eqref{eq:cga-definition},
\[
 c=\frac{2P}{x^{k-h}L},\qquad
 c_h=\frac{2Q}{x^{k-h-1}L},\qquad
 u=\frac{2S}{x^{h-1}L},\qquad
 b_{h+1}=\frac{2T}{x^hL},\qquad
 g=\frac{F_hA_{k,h}}{DF_{h-1}}.
\]
Moreover, the two identities in \eqref{eq:cu-relations} give
\[
 \frac{(cu-\alpha\beta)^2}{cg}
 =\frac1c(cu-\alpha\beta)\frac{cu-\alpha\beta}{g}
 =\frac{A_{k,h}F_{h-1}}{DcF_h}.
\]
Substituting these expressions into \eqref{eq:Schur-determinant-formula}
and multiplying by $DF_h$ yields
\begin{align}
 DF_h\det S_h^\eps
={}&4xF_hQS+4F_hQT-\frac{x^kF_hQ}{P}
     -\frac{xA_{k,h}F_{h-1}Q}{P}-x^{k+1}F_h.
\label{eq:appendix-nonterminal-expanded}
\end{align}

To reduce the right-hand side, set $s=(-1)^h$ and $t=(-1)^{k-h}$.
The recurrence \eqref{eq:R-recurrence} gives $Q=xP+s$ and $S=xT+t$.
Moreover, \eqref{eq:R-definition} gives $x^h=(x+1)P+s$; together with
\eqref{eq:F-definition}, this yields $F_h=(x-1)P+s$ and
$xF_{h-1}=(x-1)P-s$.  Finally, \eqref{eq:R-recurrence} and
\eqref{eq:F-definition} give
$F_j=R_j-R_{j-1}=(x-1)R_{j-1}+(-1)^j$ for $j\geq1$, so
$F_{h+1}=(x-1)Q-s$.

Put $X=x^k$. By \eqref{eq:tent-product-identity}, $A:=A_{k,h}=X-4PS$; let $\mathcal R$ denote the
right-hand side of \eqref{eq:appendix-nonterminal-expanded}.  We compare
$\mathcal R$ with $F_{h+1}(X-4QT)$.  Thus
\begin{align*}
 \mathcal R-F_{h+1}(X-4QT)
 ={}&4xF_hQS+4F_hQT-\frac{XF_hQ}{P}
      -\frac{xAF_{h-1}Q}{P}-xXF_h\\
    &-XF_{h+1}+4F_{h+1}QT.
\end{align*}
Using $A=X-4PS$ and $xF_{h-1}=(x-1)P-s$, we first expand
\begin{align*}
 -\frac{xAF_{h-1}Q}{P}
 &=-\frac{(X-4PS)[(x-1)P-s]Q}{P}\\
 &=-\frac{X[(x-1)P-s]Q}{P}+4SQ[(x-1)P-s].
\end{align*}
Using the identity $F_h=(x-1)P+s$ established above,
\[
 -\frac{XF_hQ}{P}=-\frac{X[(x-1)P+s]Q}{P}.
\]
This combines with the first term from the preceding expansion to give
$-2X(x-1)Q$.  Hence
\begin{align*}
 \mathcal R-F_{h+1}(X-4QT)
 ={}&4xF_hQS+4F_hQT\\
    &+4SQ[(x-1)P-s]+4F_{h+1}QT\\
    &-2X(x-1)Q-xXF_h-XF_{h+1}.
\end{align*}
Using $F_h=(x-1)P+s$, we obtain the following identity for the terms
containing $S$:
\[
 4xF_hQS+4SQ[(x-1)P-s]
 =4(x-1)SQ[(x+1)P+s].
\]
Similarly, since $F_h=(x-1)P+s$ and $F_{h+1}=(x-1)Q-s$, we have
$F_h+F_{h+1}=(x-1)(P+Q)$, and hence
$4F_hQT+4F_{h+1}QT=4(x-1)QT(P+Q)$.  For the remaining $X$-terms, using
$F_h=(x-1)P+s$, $F_{h+1}=(x-1)Q-s$, and $Q=xP+s$, we obtain
$xF_h+F_{h+1}=(x-1)(xP+Q+s)=2(x-1)Q$.  Hence
\[
 -2X(x-1)Q-xXF_h-XF_{h+1}=-4X(x-1)Q.
\]
Combining the three groups gives
\[
 \mathcal R-F_{h+1}(X-4QT)
 =4(x-1)Q\left(S[(x+1)P+s]+T(P+Q)-X\right).
\]
Now $Q=xP+s$ gives $P+Q=(x+1)P+s$, while $S=xT+t$ gives
$S+T=(x+1)T+t$.  Hence the expression in parentheses factors as
\[
 [(x+1)P+s][(x+1)T+t]-X.
\]
By \eqref{eq:R-definition}, $(x+1)P+s=x^h$ and
$(x+1)T+t=x^{k-h}$, so this factor is $x^h x^{k-h}-x^k=0$.
Thus
\[
 \mathcal R=F_{h+1}(x^k-4QT).
\]
Consequently,
\[
 x^{k-1}L^2F_h\det S_h^\eps
 =F_{h+1}\bigl(x^k-4R_hR_{k-h-1}\bigr)
 =F_{h+1}A_{k,h+1},
\]
which is the identity used in the nonterminal Schur recurrence.

\section{Endpoint block identities}
\label{app:endpoints}

This appendix records the endpoint calculations used in
Section~\ref{sec:parity}.  We retain the notation of that section, including
the convention of suppressing the dependence on $x$ when no confusion can
arise; in particular, $L=x-1$, $\alpha=x/L$, $\beta=1/L$, and
$\eps\in\{1,-1\}$.  The same nonterminal recurrence applies in the odd
and even cases.

Let $k=2r+1$ with $r\geq1$, and assume $x\geq5$.  From
\eqref{eq:G-bands}, \eqref{eq:CBCT}, and
\eqref{eq:common-congruent-matrix}, under the endpoint pairing
$(0,k),(1,k-1),\ldots,(r,r+1)$, the last diagonal block of
$\mathcal Q_k^\eps$ is
\begin{equation}
 D_r^\eps=
 \begin{pmatrix}
 b_r+b_{r+1}-\eps\alpha&
 -b_{r+1}+\eps(\alpha+\beta)\\
 -b_{r+1}+\eps(\alpha+\beta)&
 b_{r+1}+b_{r+2}-\eps\beta
 \end{pmatrix}.
\end{equation}
Its coupling to the preceding block is \eqref{eq:Eh-block} with
$h=r-1$.  Put $c=c_{r-1}$, $g=g_{r-1}$,
$a=a_{r-1}$, and $u=b_r$.  By Lemma~\ref{lem:universal-schur},
$S_{r-1}^\eps$ is invertible.  By \eqref{eq:b-entries} and \eqref{eq:cga-definition},
$b_{r+2}=c$.  Hence \eqref{eq:generic-inverse-coupling} and
\eqref{eq:generic-column-cancellation} give
\begin{equation}
 (E_{r-1}^\eps)^{\mathsf T}
 (S_{r-1}^\eps)^{-1}E_{r-1}^\eps
 =
 \begin{pmatrix}
 \rho&\eps\beta\\
 \eps\beta&c
 \end{pmatrix},
 \qquad
 \rho=\frac{cu^2-2u\alpha\beta+a\beta^2}{g}.
\end{equation}
Therefore the Schur complement of $S_{r-1}^\eps$ is
\begin{equation}\label{eq:appendix-odd-Schur}
 \begin{aligned}
 S_r^\eps
 &=
 D_r^\eps-
 (E_{r-1}^\eps)^{\mathsf T}
 (S_{r-1}^\eps)^{-1}E_{r-1}^\eps\\
 &=
 \begin{pmatrix}
 b_r+b_{r+1}-\eps\alpha-\rho&
 -b_{r+1}+\eps\alpha\\
 -b_{r+1}+\eps\alpha&
 b_{r+1}-\eps\beta
 \end{pmatrix}.
 \end{aligned}
\end{equation}

For $\eps=1$, \eqref{eq:b-entries} and
\eqref{eq:F-definition} give
\begin{equation}\label{eq:appendix-odd-plus-entries}
 \begin{aligned}
 (S_r^+)_{22}
 &=
 b_{r+1}-\beta
 =
 \frac{2R_r-x^r}{x^r(x-1)}
 =
 \frac{F_r}{x^r(x-1)},\\
 (S_r^+)_{12}
 &=
 -b_{r+1}+\alpha
 =
 \frac{x^{r+1}-2R_r}{x^r(x-1)}
 =
 \frac{F_{r+1}}{x^r(x-1)}.
 \end{aligned}
\end{equation}
The last equality in the second line uses $F_{r+1}=x^{r+1}-2R_r$,
obtained from \eqref{eq:R-recurrence} and \eqref{eq:F-definition}.

Since $k$ is odd, \eqref{eq:parity-identity} gives
$(\Delta_k+T_k)s_k=0$.  By
\eqref{eq:combined-congruence} and
\eqref{eq:common-congruent-matrix},
$\mathcal Q_k^+$ is therefore singular.  The simultaneous row-and-column permutation used in the endpoint pairing
and each successive Schur elimination are implemented by invertible matrices
whose determinants have absolute value one, and hence preserve the determinant.
After the successive eliminations, the resulting block-diagonal matrix is
\[
S_0^+\oplus S_1^+\oplus\cdots\oplus S_r^+.
\]
Therefore,
\begin{equation}
\det\mathcal Q_k^+
=
\left(\prod_{h=0}^{r-1}\det S_h^+\right)\det S_r^+
=
\left(\prod_{h=0}^{r-1}g_h\right)\det S_r^+.
\label{eq:appendix-odd-det-factorisation}
\end{equation}
Since $g_h>0$ for $0\leq h<r$, it follows that
$\det S_r^+=0$.  Moreover, $F_r>0$, so
$(S_r^+)_{22}=F_r/[x^r(x-1)]$ is nonzero.  Hence the zero determinant,
together with \eqref{eq:appendix-odd-plus-entries}, determines the
remaining entry and gives the matrix in \eqref{eq:odd-final-plus}.

Since $c$, $g$, $a$, and $u$, and hence $\rho$, are independent of
$\eps$, comparing \eqref{eq:appendix-odd-Schur} at $\eps=-1$ and
$\eps=1$ gives
\begin{equation}\label{eq:appendix-odd-sign-change}
 S_r^-
 =
 S_r^+
 +
 2\begin{pmatrix}
 \alpha&-\alpha\\
 -\alpha&\beta
 \end{pmatrix}.
\end{equation}
By \eqref{eq:appendix-odd-plus-entries} and
\eqref{eq:appendix-odd-sign-change}, the $(2,2)$ entry simplifies by
\eqref{eq:F-definition}, while \eqref{eq:R-recurrence} and
\eqref{eq:F-definition} give $F_{r+1}=x^{r+1}-2R_r$ for the off-diagonal entry. Thus
\begin{equation}\label{eq:appendix-odd-minus-entries}
 \begin{aligned}
 (S_r^-)_{22}&=\frac{x^r+2R_r}{x^r(x-1)}=C_r^*,\\
 (S_r^-)_{12}&=-\frac{x^{r+1}+2R_r}{x^r(x-1)}=B_r^*.
 \end{aligned}
\end{equation}
Here $C_r^*$ and $B_r^*$ are defined in \eqref{eq:odd-final-parameters}.

Furthermore, \eqref{eq:R-recurrence} and
\eqref{eq:F-definition} give
$F_{r+1}+F_r=x^r(x-1)$.  Using this identity together with
\eqref{eq:appendix-odd-plus-entries} and \eqref{eq:appendix-odd-sign-change}
yields
\begin{equation}
 \det S_r^-
 =
 \frac{
 2\bigl(x^{2r}(x-1)-2x^rF_r-F_r^2\bigr)
 }{
 x^r(x-1)F_r
 }.
\end{equation}
Since $F_r=2R_r-x^r$ by \eqref{eq:F-definition}, the numerator in
parentheses satisfies
\begin{equation}
 x^{2r}(x-1)-2x^rF_r-F_r^2
 =x^{2r+1}-4R_r^2
 =A_{k,r+1},
\end{equation}
where the last equality is \eqref{eq:tent-product-identity} with
$k=2r+1$ and $h=r+1$.  Hence
\begin{equation}\label{eq:appendix-odd-minus-det-final}
 \det S_r^-
 =
 \frac{2A_{k,r+1}}{x^r(x-1)F_r}
 =
 G_r^*,
\end{equation}
where $G_r^*$ is defined in \eqref{eq:odd-final-parameters}.  Since
$R_r>0$, we have $C_r^*>0$.  Thus
\eqref{eq:appendix-odd-minus-entries} and
\eqref{eq:appendix-odd-minus-det-final} determine the remaining entry and
give \eqref{eq:odd-final-minus}.

For $k=2r$ with $r\geq1$, assume $x\geq3$.  The paired order is
$(0,k),(1,k-1),\ldots,(r-1,r+1),r$.  From
\eqref{eq:G-bands}, \eqref{eq:CBCT}, and
\eqref{eq:common-congruent-matrix}, before eliminating the final pair,
the $2\times1$ off-diagonal block between that pair and the central
singleton and the $1\times1$ central diagonal block are, respectively,
\begin{equation}\label{eq:appendix-even-endpoint}
 \begin{pmatrix}
 -b_r-\eps\alpha\\
 -b_{r+1}-\eps\beta
 \end{pmatrix},
 \qquad
 b_r+b_{r+1}+\eps(\alpha+\beta).
\end{equation}
With the abbreviations in \eqref{eq:even-end-abbreviations},
\eqref{eq:b-entries} and \eqref{eq:cga-definition} give $b_{r+1}=c$, so
the two expressions in \eqref{eq:appendix-even-endpoint} are precisely
$v_\eps$ and $\eta_\eps$ in \eqref{eq:even-final-data}.

Write $S=S_{r-1}^\eps$ and decompose
$v_\eps=e+d$, where
$e=(-\eps\alpha,-c)^{\mathsf T}$ and
$d=(-B,-\eps\beta)^{\mathsf T}$.  Since $S$ is a nonterminal
positive-definite Schur block, it is invertible.  By
\eqref{eq:generic-column-cancellation},
$S^{-1}e=(0,-1)^{\mathsf T}$, and hence
$e^{\mathsf T}S^{-1}e=c$ and
$d^{\mathsf T}S^{-1}e=\eps\beta$.
Using the inverse in \eqref{eq:generic-inverse-coupling} and
$\det S=g$ also gives
$d^{\mathsf T}S^{-1}d
=\beta^2/c+(cB-\alpha\beta)^2/(cg)$.
Consequently,
\begin{equation}\label{eq:appendix-even-quadratic}
 v_\eps^{\mathsf T}S^{-1}v_\eps
 =
 c+2\eps\beta
 +\frac{\beta^2}{c}
 +\frac{(cB-\alpha\beta)^2}{cg}.
\end{equation}
Substituting \eqref{eq:appendix-even-quadratic} and
$\eta_\eps=B+c+\eps(\alpha+\beta)$ from
\eqref{eq:even-final-data} into the scalar Schur complement, and using
$\alpha-\beta=1$, gives \eqref{eq:even-sigma}.  Thus the nonterminal recurrence is common to the odd and even cases;
in the even case, the final scalar Schur complement depends on $\eps$ only
through the additive term $\eps$ in \eqref{eq:even-sigma}.

\section{Finite-factor Kronecker identities}
\label{app:finite-kron}

The basic Kronecker-product properties used below, including the block
definition, the transpose identity, the mixed-product property, and
associativity, are standard; see \cite[Section~4.2]{hornjohnson1991} and
\cite[Section~1, pp.~85--86]{vanloan2000}.  In particular, for $A=(a_{ij})$ and a matrix
$B$, the two-factor block definition is
\[
 A\otimes B=[a_{ij}B].
\]
The remaining two-factor formulas used below follow directly from these
properties.  For completeness, we record the finite-factor forms used in
Sections~\ref{sec:transforms} and~\ref{sec:tensor-bounds}.

\begin{lemma}
Let $m\geq1$.  Let $A_\nu,B_\nu$ be matrices and $x_\nu,y_\nu$ real
column vectors, with all dimensions compatible in the identities below; let
$c_\nu$ be scalars and $D_\nu$ positive diagonal matrices.  For
\eqref{eq:finite-trace}, assume additionally that each $A_\nu$ is square.  Then
\begin{align}
 &(A_m\otimes\cdots\otimes A_1)(B_m\otimes\cdots\otimes B_1)
   =(A_mB_m)\otimes\cdots\otimes(A_1B_1),
 \label{eq:finite-mixed-product}\\
 &(A_m\otimes\cdots\otimes A_1)^{\mathsf T}
   =A_m^{\mathsf T}\otimes\cdots\otimes A_1^{\mathsf T},
 \label{eq:finite-transpose}\\
 &(A_m\otimes\cdots\otimes A_1)(x_m\otimes\cdots\otimes x_1)
   =(A_mx_m)\otimes\cdots\otimes(A_1x_1),
 \label{eq:finite-action}\\
 &\langle x_m\otimes\cdots\otimes x_1,
          y_m\otimes\cdots\otimes y_1\rangle
   =\prod_{\nu=1}^{m}\langle x_\nu,y_\nu\rangle,
 \label{eq:finite-inner-product}\\
 &(c_mx_m)\otimes\cdots\otimes(c_1x_1)
   =\left(\prod_{\nu=1}^{m}c_\nu\right)
    (x_m\otimes\cdots\otimes x_1),
 \label{eq:finite-scalars}\\
 &\tr(A_m\otimes\cdots\otimes A_1)
   =\prod_{\nu=1}^{m}\tr(A_\nu).
 \label{eq:finite-trace}
\end{align}
Moreover, $D_m\otimes\cdots\otimes D_1$ is positive diagonal and, for
$\alpha\in\{1/2,-1/2\}$,
\begin{equation}\label{eq:finite-diagonal-power}
 (D_m\otimes\cdots\otimes D_1)^\alpha
 =D_m^\alpha\otimes\cdots\otimes D_1^\alpha.
\end{equation}
If $x_\nu\in\R^{a_\nu+1}$ and the multi-index ordering of
\eqref{eq:vec-convention} is used, then
\begin{equation}\label{eq:finite-coordinate}
 (x_m\otimes\cdots\otimes x_1)_{\mathbf i}
 =\prod_{\nu=1}^{m}(x_\nu)_{i_\nu}
\end{equation}
for every $\mathbf i=(i_1,\ldots,i_m)\in\mathcal I$.  Finally, if
$A_\nu$ has size $r_\nu\times s_\nu$ for $1\leq\nu\leq m$, then
\begin{equation}\label{eq:finite-dimension}
 A_m\otimes\cdots\otimes A_1
 \text{ has size }
 \left(\prod_{\nu=1}^{m}r_\nu\right)
 \times
 \left(\prod_{\nu=1}^{m}s_\nu\right).
\end{equation}
\end{lemma}

\begin{proof}
The case $m=1$ is immediate, so assume $m\geq2$.  We first prove
\eqref{eq:finite-mixed-product}, which also illustrates the induction used below.  For $m=2$, it is exactly the standard two-factor
mixed-product identity
\[
 (A_2\otimes A_1)(B_2\otimes B_1)
 =(A_2B_2)\otimes(A_1B_1).
\]
Assume that \eqref{eq:finite-mixed-product} holds for $m-1$ factors.  By
associativity,
\[
 A_m\otimes\cdots\otimes A_1
 =A_m\otimes(A_{m-1}\otimes\cdots\otimes A_1),
\]
and similarly for the $B_\nu$.  Applying the two-factor mixed-product
identity and then the induction hypothesis gives
\begin{align*}
 &(A_m\otimes\cdots\otimes A_1)(B_m\otimes\cdots\otimes B_1)\\
 &\qquad=(A_mB_m)\otimes
 \bigl[(A_{m-1}\otimes\cdots\otimes A_1)
       (B_{m-1}\otimes\cdots\otimes B_1)\bigr]\\
 &\qquad=(A_mB_m)\otimes\cdots\otimes(A_1B_1),
\end{align*}
which proves \eqref{eq:finite-mixed-product}.

The proofs of \eqref{eq:finite-transpose}, \eqref{eq:finite-scalars},
\eqref{eq:finite-trace}, and \eqref{eq:finite-dimension} have the same inductive
form.  Associativity separates the factor with index $m$.  For
\eqref{eq:finite-transpose}, one applies the standard two-factor transpose
identity; for \eqref{eq:finite-scalars}, \eqref{eq:finite-trace}, and
\eqref{eq:finite-dimension}, the corresponding two-factor conclusions follow
directly from the block definition.  The induction hypothesis then gives the
stated finite-factor conclusions.
Viewing each $x_\nu$ as a one-column matrix, we obtain
\eqref{eq:finite-action} as a special case of \eqref{eq:finite-mixed-product}.

For \eqref{eq:finite-inner-product}, first consider two factors.  By the
definition of the Euclidean inner product, the two-factor transpose identity,
and the mixed-product identity,
\begin{align*}
 \langle x_2\otimes x_1,y_2\otimes y_1\rangle
 &=(x_2\otimes x_1)^{\mathsf T}(y_2\otimes y_1)\\
 &=(x_2^{\mathsf T}\otimes x_1^{\mathsf T})(y_2\otimes y_1)\\
 &=(x_2^{\mathsf T}y_2)(x_1^{\mathsf T}y_1)\\
 &=\langle x_2,y_2\rangle\langle x_1,y_1\rangle.
\end{align*}
Assume now that the identity holds for $m-1$ factors.  By associativity and
the two-factor identity just proved,
\begin{align*}
 &\left\langle x_m\otimes\cdots\otimes x_1,
 y_m\otimes\cdots\otimes y_1\right\rangle\\
 &\qquad=\langle x_m,y_m\rangle
 \left\langle x_{m-1}\otimes\cdots\otimes x_1,
 y_{m-1}\otimes\cdots\otimes y_1\right\rangle\\
 &\qquad=\prod_{\nu=1}^{m}\langle x_\nu,y_\nu\rangle,
\end{align*}
which proves \eqref{eq:finite-inner-product}.

We next prove \eqref{eq:finite-diagonal-power}.  Let
$D=\operatorname{diag}(d_i)$ and $E=\operatorname{diag}(e_j)$ be positive
diagonal matrices.  By the two-factor block definition,
\[
 D\otimes E=\operatorname{diag}(d_i e_j).
\]
Hence $D\otimes E$ is positive diagonal and, for
$\alpha\in\{1/2,-1/2\}$,
\[
 (D\otimes E)^\alpha
 =\operatorname{diag}\bigl((d_i e_j)^\alpha\bigr)
 =\operatorname{diag}(d_i^\alpha e_j^\alpha)
 =D^\alpha\otimes E^\alpha.
\]
By associativity, the two-factor result above, and induction on the number of
factors, $D_m\otimes\cdots\otimes D_1$ is positive diagonal and
\eqref{eq:finite-diagonal-power} follows.

For \eqref{eq:finite-coordinate}, first consider two factors and
view $x_2$ and $x_1$ as column matrices.  The two-factor block definition
shows that the block indexed by $i_2$ is $(x_2)_{i_2}x_1$.  Consequently,
\[
 (x_2\otimes x_1)_{\iota(i_1,i_2)}
 =(x_2)_{i_2}(x_1)_{i_1},
 \qquad
 \iota(i_1,i_2)=i_1+(a_1+1)i_2,
\]
which is exactly the two-factor instance of the convention in
\eqref{eq:vec-convention} that the first coordinate varies fastest.
With the ordering fixed by \eqref{eq:vec-convention}, associativity, the
two-factor coordinate formula above, and induction on the number of factors give
\[
 (x_m\otimes\cdots\otimes x_1)_{\mathbf i}
 =\prod_{\nu=1}^{m}(x_\nu)_{i_\nu},
\]
proving \eqref{eq:finite-coordinate}.  This completes the proof.
\end{proof}

\section*{Statements and Declarations}

\noindent\textbf{Funding.}
This work was supported by the National Natural Science Foundation of China under Grant No.~12071351.

\medskip
\noindent\textbf{Competing interests.}
The authors declare that they have no known competing financial interests or personal relationships that could have appeared to influence the work reported in this paper.

\medskip
\noindent\textbf{Data availability.}
No data were used for the research described in this article.

\medskip
\noindent\textbf{AI use declaration.}
After independently developing the mathematical results and proofs presented in this work, the authors used ChatGPT (OpenAI) to assist with drafting and polishing portions of the exposition. The authors reviewed and edited all AI-assisted text and take full responsibility for the manuscript.

\end{document}